\documentclass[11pt,reqno]{amsart}
\usepackage{amsmath,amsfonts,amssymb,amscd,amsthm,amsbsy}
\usepackage[dvips]{graphicx}
\usepackage{mathabx}
\usepackage{comment}
\usepackage[dvips]{color}
\usepackage[dvipsnames]{xcolor}
\usepackage{fancyhdr}
\usepackage{hyperref}
\numberwithin{equation}{section}
\newtheorem{theorem}{Theorem}
\newtheorem{meta-thm}[theorem]{Meta-Theorem}
\newtheorem{lemma}[theorem]{Lemma}
\newtheorem{cor}[theorem]{Corollary}

\newtheorem{remark}[theorem]{Remark}

\newcommand{\R}{\mathbb{R}}
\newcommand{\N}{\mathbb{N}}

\newcommand{\T}{\mathbb{T}}

\def\ep{\varepsilon}

\def\C{{\mathcal C}}

\def\E{{\mathcal E}}

\def\eps{\varepsilon}

\begin{document}
\title[state-dependent delay perturbation to an ODE]
{Parameterization method for state-dependent delay perturbation of 
an ordinary differential equation} 
\author[J. Yang]{Jiaqi Yang}
\address{
School of Mathematics,
Georgia Institute of Technology,
686 Cherry St.. Atlanta GA. 30332-0160 }
\email{jyang373@gatech.edu}

\author[J. Gimeno]{Joan Gimeno} 
\address{
Departament de Matem\`atiques i Inform\`atica,
Barcelona Graduate School of Mathematics (BGSMath),
Universitat de Barcelona (UB).
Gran Via de les Corts Catalanes 585, 08007, Barcelona, Spain.
}
\email{joan@maia.ub.es}

\author[R. de la Llave]{Rafael de la Llave}
\address{
School of Mathematics,
Georgia Institute of Technology,
686 Cherry St.. Atlanta GA. 30332-0160 }
\email{rafael.delallave@math.gatech.edu}

\thanks{
J.~Y. and R.~L. were partially supported by NSF grant DMS-1800241.
J.~G. acknowledges financial support from the Spanish Ministry of Economy and 
Competitiveness, through the Mar\`ia de Maeztu Programme for Units of 
Excellence in R\&D (MDM-2014-0445). Also Spanish grants PGC2018-100699-B-I00 
(MCIU/AEI/FEDER, UE) and the Catalan grant 2017 SGR 1374.
This research was funded by H2020-MCA-RISE \#734577 which supported 
visits of J.~G. to Georgia Inst. of Technology and of J.~Y. to Univ. of 
Barcelona. 
J.~G. thanks School of Mathematics GT for hospitality
Springs 2018 and 2019.
}


\begin{abstract}
We consider state-dependent delay equations (SDDE) obtained by adding delays
to a planar ordinary differential equation with a limit cycle. 
These situations appear in models of several physical processes, where small delay effects are added. Even if the delays are small, they
are very singular perturbations since the natural phase space of an SDDE is 
an infinite dimensional space. 

We show that the SDDE admits solutions which resemble the solutions of  the ODE. 
That is, there exist a periodic solution and a two parameter family of solutions 
whose evolution converges to the periodic solution (in the ODE case, these are 
called the isochrons). Even if the phase space of the SDDE is naturally a space of 
functions, we show that there are initial values which lead to solutions similar to that of the ODE.

The method of proof bypasses the theory of existence, uniqueness, dependence on 
parameters of SDDE. We consider the class of functions of time that have a well defined 
behavior (e.g. periodic, or asymptotic to periodic) and derive a functional 
equation which imposes that they are solutions of the SDDE.  These functional equations 
are studied using methods of functional analysis. We provide a result in 
``a posteriori'' format: Given an approximate solution of the functional equation, which 
has some good condition numbers, we prove that there is true solution close to the approximate one. 
Thus, we can use the result to validate the results of  numerical computations.
The method of proof leads also to practical algorithms. In a companion paper, 
we present the implementation details and representative results. 

One feature of the method presented here is that it allows to obtain smooth dependence on parameters for the periodic solutions and their slow stable manifolds without studying the smoothness of the flow (which seems to be problematic for SDDEs, for now the optimal result on smoothness of the flow is $C^1$). 
\end{abstract}


\maketitle

\tableofcontents
\fancyhf{}

\section{Introduction}\label{sec:intro}
Many causes in natural sciences take some time to generate an effect. 
If one incorporates this delay in the models, one is lead 
to descriptions of systems in which the derivatives of states are functions 
of the states at previous times.  These are commonly called 
delay differential equations.

In the case that the delay is constant (say $1$), one can 
prescribe the data in an interval $[-1,0]$ and then propagate 
the differential equation. This leads to a rather satisfactory
theory of existence and uniqueness and even a qualitative theory
\cite{Driver, Hale, HaleLunel, Diek}. Note that the natural phase
space is a space of functions on $[-1,0]$. This is an infinite dimensional 
space.

When the delay is not a constant and depends on 
the state, one needs to consider State-Dependent Delay Equations
(SDDE for short). In contrast with the constant delay case, 
the mathematical theory of SDDE has complications. The paper \cite{Walt} made important progress for the appropriate phase space for SDDE. 
We refer to \cite{hart} for a very comprehensive survey 
of the mathematical theory and the (rather numerous) applications. 

In this paper, we consider a simple model (two-dimensional 
ordinary differential equation with a limit cycle) 
and show that all solutions close to the limit cycle present in this model 
persist (in some appropriate sense) when we add a state-dependent delay perturbation.  
  Models of the form considered in this paper appear in several concrete problems in the natural 
sciences (circuits, neuroscience, and population dynamics), see \cite{hart}. 

The result is subtle to formulate 
since the perturbation of adding a state-dependent delay is very singular, it changes the nature of the equation: the unperturbed 
case is an ODE and the perturbed case is an infinite-dimensional 
problem. The basic idea is that 
we establish the existence of 
some finite-dimensional families of solutions (in the phase space of 
the SDDE), which resemble (in an appropriate sense)
the solutions of the original ODE. 
This allows to establish many other properties (e.g. dependence on parameters) 
which may be false for general solutions of SDDE. 
We hope that the method can be extended
in several directions. For example, we hope to produce higher dimensional families, 
families with other behaviors, and more complicated models.   The conjectural picture
that appears is that in SDDEs, even if the dynamics in a full Banach space of 
solutions is problematic, one can find  a very rich set of solutions organized in 
families even if the families may not fit together well and leave gaps, so 
that a  general theory may have problems \cite{CurrieS}. 

\subsection{Overview}

Let us start by an informal overview of the method.  It is known that
in a neighborhood of a limit cycle of a 2-dimensional ODE, we can find
$K: \T \times [-1,1]\to\R^2$, and $\omega_0$ and $\lambda_0$ in such a
way that for any $\theta, s$, the function given by
\begin{equation}\label{parameterization} 
x(t) = K( \theta + \omega_0 t, s e^{\lambda_0 t} )
\end{equation} 
solves the ODE, see \cite{HL13}. The fact that all the functions of the form \eqref{parameterization} 
are solutions of the original ODE is equivalent to a functional equation 
for $K$, $\omega_0$ and $\lambda_0$. Efficient methods to study the
resulting functional equation were presented in \cite{HL13}. We will, henceforth, assume that $K$, $\omega_0$, $\lambda_0$ are known.

Similarly, for the perturbed case, when we impose that for fixed
$\theta$, $s$ the function of the form
\begin{equation}\label{parameterization2}
x(t) = K\circ W(  \theta + \omega t, s e^{\lambda t} )
\end{equation}
is a solution of our delay differential equation, we obtain
a functional equation for $W$, $\omega$, $\lambda$ (see \eqref{inv}), which we call ``invariance equation".  Note that the 
unknowns in \eqref{inv} are the embedding $W$ and the numbers 
$\omega, \lambda$.    

Our goal will be to solve \eqref{inv} using techniques of functional 
analysis. The equation is rather degenerate and our treatment has 
several steps. We first find some asymptotic expansions  in powers of $s$ to a finite order, and 
then, we formulate a fixed point problem for the remainder. 
Due to the delay, information at previous times is needed. We anticipate a technical problem is that the domain of 
definition of the unknown have to depend on the details of the unknown. 
Similar problems appear in the theory of center manifolds \cite{Carr}. Here we have to resort to cut-offs and extensions. After this process, we get a prepared equation, \eqref{invv}, which has the same format as equation \eqref{inv}, and agrees with equation \eqref{inv} in a neighborhood. Solutions of the prepared equation which stay in the neighborhood will be solutions of the original problem.

The main results of this paper is Theorem \ref{thm:all}, which establish 
that with respect to some condition numbers of the problem, verified for small enough $\ep$, given an approximate solution of the extended invariance equation
\eqref{invv} of the problem, one obtain a true solution nearby. (This is sometimes referred as \emph{``a posteriori''} format.)

As in the case of center manifolds, the family of solutions found to the original problem may depend on the extension considered.

\subsection{Some comments on the results}

In a geometric language, we can describe our procedure as saying that 
we are finding an embedding of the phase space of the ODE into the phase space of 
the SDDE in such a way that the range of the embedding is foliated by solutions of
the SDDE and that the flow in this manifold is similar to the flow of the ODE. 
Note that this bypasses the need of developing a general theory of 
solutions of the SDDE. We only construct a 2-D manifold of solutions of 
the SDDE. For these solutions, it is possible 
to discuss comfortably many desirable properties such 
as smooth dependence on the model, etc.

Philosophies similar to that of this paper (finding 
solutions of functional equations that inply the existence of
solutions of special kinds)  have already been used in \cite{HR, HR2, Livia19}
to study quasi-periodic solutions of SDDE. 
For constant delay equations, we can find \cite{Lessard10, KissL12} for the study of periodic solutions. The paper \cite{KissL17a} studies specific  models similar to 
ours for constant delay perturbations.  
The paper 
\cite{LiL} studies  quasi-periodic solutions analytically, \cite{GroothedeMJ17} 
studies numerically  unstable manifolds near fixed points. 
The papers \cite{Sieber}, \cite{Hum}, \cite{HumphriesBCHS16,MagpantayKW14}
study normal forms  and numerical 
computations of periodic and quasi-periodic solutions of SDDEs  and 
obtain bifurcations and numerical solutions.
Even if the evolutions of the SDDEs considered above are difficult to define
as smooth evolutions, we believe that the results above 
can be understood as suggesting the existence of a subsystem of 
the evolution which indeed experiences bifurcations. The careful numerical solutions of 
\cite{Hum} can presumably be validated.

By solving the invariance equation, \eqref{invv}, one actually obtains a parameterization of the limit cycle and its isochrons (2-dimensional slow stable manifold of the limit cycle). In other words,
 $K\circ W(\theta, 0)$ parameterizes the limit cycle, and for fixed $\theta$, we
have $K\circ W( \theta,s)$ parameterizes the local slow stable manifold of
the point $K\circ W( \theta,0)$ on the limit cycle. We remark that in
some previous work, Chapter 10 of \cite{HaleLunel}, persistence of limit cycles were
studied with a different method in the setting of retarded functional
differential equations(RFDE). They have also studied
infinite-dimensional stable manifolds of periodic orbits of RFDE. In
this paper, we study SDDE, and get a parameterization of the
submanifold of the infinite-dimensional stable manifold, which corresponds to
the eigenvalue of the time-T map with largest modulus. In this sense,
we think that the manifold in this paper is practically more relevant
than the infinite-dimensional manifolds. For a more detailed
comparison of the results of this paper with the study of SDDE as
evolutionary equations, see Section~\ref{ssc:evolution}.

Of course, the notions of approximate solutions and  that of solutions close 
to the approximate ones, requires to specify a norm in space of functions. 
In \cite{HL13}, it was natural to specify analytic norms. In this 
paper, however, we use spaces of finitely differentiable functions. Indeed,
we conjecture that the solutions we produce are not more 
than finitely differentiable.

The a-posteriori format of Theorem~\ref{thm:all} allows us to validate
approximate solutions produced even by non-rigorous methods.  In that
respect, we note that the related paper \cite{Joan19} develops
numerical methods that produce approximate solutions. Some papers 
that study formal expansions in the delay are 
\cite{CasalFreed} for periodic solutions and bifurcations, 
mostly with constant delay, and  \cite{Livia19} which studies periodic 
and quasiperiodic solutions for SDDE (and even 
more general models such as those appearing in  electrodynamics). 

Using
Theorem~\ref{thm:all}, we obtain that the numerical solutions produced
in \cite{Joan19}, have true solutions nearby and that the formal
expansion produced in \cite{Livia19} are not just formal expansions
but are asymptotics to a true solution.  For an earlier example or
related philosophies, we mention that asymptotic expansions for
equations with small constant delay was produced and validated in the
paper \cite{Chiconedelay}.

A rather subtle point is that we do not obtain uniqueness of the
solution. The reason is that the nature of the problem involves cutting off
the perturbation and the solution produced may depend on the cut-off
function used.  Both the finite regularity and the lack of uniqueness
due to the introduction of a cut-off are reminiscent to effects found
in the study of center manifolds \cite{Carr,Lan}.  Of course, since
one of the goals of the paper is to remedy the paucity of solutions of
SDDEs, having many solutions is a feature not a bug. The dependence
of the solutions in the cut-off has to be small as the delay tends to
zero (note that the asymptotic expansions in \cite{Livia19} do not
depend on the cut-off), but we expect that they are small in other
senses similar to the situation in center manifolds \cite{Sijbrand}.
We will not formulate here results making precise this intuition.

We hope that the methods of this paper can be extended to prove the
existence of other finite-dimensional families of solutions that are
not close to families of solutions of the unperturbed ODE.

\subsection{Organization of the paper} 

We  introduce the problem and formulate the equations to be solved in
section \ref{sc:formu}. In  Section \ref{sc:basic} we present 
some notations and some classical results in functional analysis which
will be used in the proof. We state our main results in section
\ref{sc:mainr}. We give an overview of the proof in section
\ref{sc:overpf}. Detailed proofs of the results are given in
section \ref{sc:pf}.

\section{Formulation of the problem} \label{sc:formu}

We consider an ordinary differential equation in the plane
\begin{equation} \label{unper}
\dot x(t) = X_0( x(t)),
\end{equation}
where $x(t)\in \R^2$, $X_0: \R^2\to \R^2$ is analytic. We assume above equation \eqref{unper} admits a limit cycle. 
Clearly, there is a two dimensional family of solutions to 
this ordinary differential equation. This family can be parameterized
e.g. by the initial conditions, but as we will see, there are more efficient 
parameterizations near the limit cycle. 

The goal of this paper is to study a state-dependent delay equation
that is a  ``small'' modification of equation \eqref{unper} in 
which we add some small term for the derivative 
that depends on some previous time. 
Adding some dependence on the solution at previous times,  arises
naturally in many problems. Limit cycles appear in feedback loops and
if the feedback loops have a delayed effect, which depends on the
present state, to incorporate them in the model, we are lead to:

\begin{equation} \label{SDDE}
  \dot x(t) = X(x(t), \varepsilon x(t-r(x(t)))), \qquad 0\leq \varepsilon \ll 1
\end{equation}
Where $x(t)\in\mathbb{R}^2$, $X:\mathbb{R}^2\times \mathbb{R}^2\to \mathbb{R}^2$ is analytic, the state-dependent delay function $r:\mathbb{R}^2\to [0,h]$ is as smooth as we need. The equation \eqref{SDDE} is
formally a perturbation of \eqref{unper} with
$X(x, 0) = X_0(x)$. 

We can rewrite \eqref{SDDE} as
\begin{equation}\label{SDDE2}
 \dot x(t) = X(x(t),0) + \varepsilon P(x(t), x(t-r(x(t))),\varepsilon), 
\end{equation} 
where we define,
 \[
  \varepsilon P(x(t), x(t-r(x(t))),\varepsilon)= X(x(t), \varepsilon x(t-r(x(t)))-X(x(t),0).
 \]

The question we want to address in this paper is to find 
a two dimensional family of solutions of \eqref{SDDE}, 
which resembles the two dimensional family of solutions of
\eqref{unper}.  This is a much simpler problem than developing a
general theory of existence of solutions of an SDDE, which is 
a rather difficult problem. Nevertheless, persistence of some family of solutions is of physical interest.

Note that, when $\varepsilon > 0$ the equation \eqref{SDDE2} is
an SDDE, which is an equation of a very different nature from 
the equation when $\varepsilon = 0$, which is an ODE. Hence, we 
are facing a very singular perturbation in which the nature
of the problem changes drastically from an ODE -- whose phase 
space is $\mathbb{R}^2$ to an SDDE -- whose natural phase space is 
a space of functions.  The precise meaning of the continuation of 
the unperturbed solutions into solutions of the perturbed problem 
is somewhat subtle. 

\subsection{Limit cycles and isochrons for ODEs}\label{sc:lciso}

Under our assumption, there exists a limit cycle in the unperturbed
equation \eqref{unper}. In a neighborhood of the limit cycle (stable periodic orbit), points
have asymptotic phases(see \cite{Win,Guc}).  The points sharing the
same asymptotic phase as point $p$ on the limit cycle is the stable
manifold for point $p$. The stable manifold of the limit cycle is
foliated by the stable manifolds for points on the limit cycle
(sometimes referred as stable foliations). The stable manifolds for
points on the limit cycle are also called isochrons in the biology
literature, see \cite{Guc}, \cite{Win}.

According to \cite{HL13}, we can find a parameterization of the limit 
cycle and the isochrons in a neighborhood of the limit cycle. More precisely, there exists real numbers $\omega_0>0$, $\lambda_0<0$, and an analytic local 
diffeomorphism $K:{\mathbb{T}}\times[-1,1]\to \mathbb{R}^2$, such that
\begin{equation}\label{invk}
X_0(K(\theta, s))=DK(\theta, s)
\begin{pmatrix}
  \omega_0   \\
  \lambda_0 s  
\end{pmatrix},
\end{equation}
where $K$ is periodic in $\theta$, i.e. $K(\theta+1,s)=K(\theta, s)$. Saying that $K$ solves
\eqref{invk} is equivalent to saying that for fixed parameters $\theta$ and $s$,  the function $x(t)=K(\theta+\omega_0t, se^{\lambda_0 t})$ solves 
\eqref{unper} for all $t$ such that $|se^{\lambda_0t}|<1$. Notice that when $s=0$, $K(\theta,0)$ parameterizes the limit cycle, and for 
a fixed $\theta$ with varying $s$, we get the local stable manifold of the point 
$K(\theta,0)$.

Note that geometrically, $K$ can be viewed as a change of coordinates, under which the original vector field is equivalent to the vector field $X_0'(\theta,s)=(\omega_0,\lambda_0s)$ in the space ${\T}\times[-1,1]$. We could have started with this vector field $X_0'$ and then added some perturbation to it. However, to keep contact with applications, we decided not to do this.

\begin{remark}\label{rmk:nonuni}
 As pointed out in \cite{HL13}, the $K$ solving \eqref{invk} can never be unique. If $K(\theta,s)$ is a solution of \eqref{invk}, then for any $\theta_0$, $b\neq 0$, $K(\theta+\theta_0, bs)$ will also be a solution of \eqref{invk}. \cite{HL13} also shows that this is the only source of non-uniqueness. We will call such $b$ scaling factor, and such $\theta_0$ phase shift. Note that by using a different $b$, we can change the domain of $K$. However, no matter how the domain changes, $s$ has to lie in a finite interval.
 \end{remark}
 
In this paper, for the equation after perturbation \eqref{SDDE}, we will show if 
$\ep$ is small enough, the limit cycle and its isochrons persist as limit cycle and its slow stable manifolds of the delayed model. We will use the name isochrons to denote the slow stable manifolds and distinguish them from the infinite dimensional stable manifolds similar to the one established by \cite{HaleLunel}. Meanwhile, we will find a parameterization of them. More precisely, we will find 
$\omega>0$, $\lambda<0$, and $W$ which maps a subset of 
${\mathbb{T}}\times \mathbb{R}$ to a subset of 
${\mathbb{T}}\times \mathbb{R}$, such that for small $s$, $K\circ W(\theta, s)$ gives us a parameterization of the limit cycle 
as well as of its isochrons in a neighborhood. We assume that $W$ can be lifted to a function from $\R^2$ to $\R^2$ (we will use the same letter to denote the function before and after the lift) which satisfies the periodicity condition: 
\begin{equation}\label{pc}
W(\theta+1, s)=W(\theta, s)+\left(\begin{smallmatrix}1\\0\end{smallmatrix}\right).
\end{equation}

We remark that $K\circ W$ being a parameterization of the limit cycle and its isochrons is the same 
as for given $\theta$, and $s$ in domain of $W$, $x(t)=K\circ W(\theta+\omega t, se^{\lambda t})$ solving $\eqref{SDDE}$ for $t\geq 0$.

\subsection{The invariance equation and the prepared invariance equation} 

Substitute $x(t)=K\circ W(\theta+\omega t, se^{\lambda t})$ into \eqref{SDDE2}, let $t=0$, with the fact that $DK$ is invertible, we get $x(t)=K\circ W(\theta+\omega t, se^{\lambda t})$ solves equation \eqref{SDDE} if and only if $W$ satisfies 
\begin{equation}\label{inv}
DW(\theta,s)
\begin{pmatrix}
  \omega   \\
  \lambda s  
\end{pmatrix}
=
\begin{pmatrix}
  \omega_0   \\
  \lambda_0 W_2(\theta, s)
\end{pmatrix}
+\varepsilon Y(W(\theta,s), \widetilde{W}(\theta, s),\varepsilon),
\end{equation}
where $W_2(\theta, s)$ is the second component of $W(\theta, s)$, $\widetilde{W}$ is the term caused by the delay:
\[
\widetilde{W}(\theta, s)=W(\theta-\omega r\circ K(W(\theta, s)),se^{-\lambda r\circ K(W(\theta, s))}),
\]
and
\[
Y(W(\theta,s), \widetilde{W}(\theta, s),\varepsilon)=(DK(W(\theta, s)))^{-1}P(K(W(\theta, s)), K(\widetilde{W}(\theta, s)), \varepsilon).
\]

Note that even if $\widetilde{W}$ is typographically convenient, $\widetilde{W}$ is a very complicated function of $W$, it involves several compositions.

Now we need to look at equation \eqref{inv} more closely and specify the domain and range of $W$. One cannot define $W$ on ${\mathbb{T}}\times [-b,b]$, where $b>0$  is a constant. Indeed, observing the second component in expression of $\widetilde{W}$, $se^{-\lambda r\circ K(W(\theta, s))}$, one will note that $|se^{-\lambda r\circ K(W(\theta, s))}|>|s|$. This will drive us out of the domain of $W$ if $W$ is defined for second component lying in a finite interval. Therefore, $W$ has to be defined for $s$ on the whole real line. So we let $W:{\mathbb{T}}\times \mathbb{R}\to {\mathbb{T}}\times \mathbb{R}$. There is another technical issue as pointed out in the following Remark \ref{exten}.

\begin{remark}\label{exten}
When $\ep$ is small, we expect $W$ to be close to the identity map. Then for $s$ far from $0$, $W(\theta, s)$ does not lie in the domain of $K$, thus the invariance equation is not well defined.
\end{remark}

Similar to the study of center manifolds. We will use cut-off functions to resolve the above issues.

We transform our original equation \eqref{inv} into a well-defined equation of the same format:
\begin{equation}\label{invv}
DW(\theta,s)
\begin{pmatrix}
  \omega   \\
  \lambda s  
\end{pmatrix}
=
\begin{pmatrix}
  \omega_0   \\
  \lambda_0 W_2(\theta, s)
\end{pmatrix}
+\varepsilon \overline{Y}(W(\theta,s), \widetilde{W}(\theta, s),\varepsilon),
\end{equation}
where $\overline{Y}$ is defined on $({\mathbb{T}}\times\R)^2\times\mathbb{R}_+$, and $\overline{r\circ K}$ is defined on ${\mathbb{T}}\times \mathbb{R}$, with slight abuse of notation, we still denote the term caused by the delay as $\widetilde{W}$:
\[
\widetilde{W}(\theta, s)=W(\theta-\omega \overline{r\circ K}(W(\theta, s)),se^{-\lambda \overline{r\circ K}(W(\theta, s))}).
\]

Following standard practice in theory of center manifolds of differential equations, see \cite{Carr}. We introduce the extensions as follows:

\begin{itemize}
\item For $r\circ K$ which is defined only on $\mathbb{T}\times[-1,1]$, we  define a function $\overline{r\circ K}$ on $\mathbb{T}\times \mathbb{R}$, which agrees with $r\circ K$ on $\mathbb{T}\times[-\frac{1}{2},\frac{1}{2}]$, and is zero outside of $\mathbb{T}\times[-1,1]$.

\item For $Y:(\mathbb{T}\times[-1,1])^2\times\mathbb{R}_+\to \mathbb{R}^2$, we define $\overline{Y}:(\mathbb{T}\times \mathbb{R})^2\times\mathbb{R}_+\to\mathbb{R}^2$, which agrees with $Y$ on the set $(\mathbb{T}\times[-\frac{1}{2},\frac{1}{2}])^2\times\mathbb{R}_+$, and is zero outside $(\mathbb{T}\times[-1,1])^2\times\mathbb{R}_+$.
\end{itemize}

To achieve above extensions, let $\phi:\mathbb{R}\to[0,1]$ be a $C^{\infty}$ cut-off function:
\begin{equation}\label{cutoff}
\phi(x)=
\begin{cases}
1 & \text{if}\quad |x|\leq\frac{1}{2},\\
0 & \text{if}\quad |x|>1.
\end{cases}
\end{equation}

We define 
\[
\overline{r\circ K}(\theta, s)=r\circ K(\theta, s)\phi(s),
\]
and,
\[
\overline{Y}(W(\theta,s), \widetilde{W}(\theta, s),\varepsilon)=Y(W(\theta,s), \widetilde{W}(\theta, s),\varepsilon)\phi(W_2(\theta, s))\phi( \widetilde{W}_2(\theta, s)).
\]

After these extensions, the main equation \eqref{inv} is turned into the well-defined equation \eqref{invv}. Note that, $\overline{Y}$, $\overline{r\circ K}$ defined above have bounded derivatives in their domain up to any order.

\begin{remark}
In the definition of cut-off function, one can let $\phi$ to vanish for $|x|>c_1$ where the constant $c_1<1$, and let $\phi=1$ for $|x|\leq c_2$ where the constant $c_2<c_1$.
\end{remark}

\begin{remark}
The use of the cut-off function here is very similar to the use of cut-offs in the study of the center manifolds in the literature, if we choose a different cut-off function $\phi$, we will possibly end up with a different $W$, which solves \eqref{invv} with the new cut-off function $\phi$. 
\end{remark}

\begin{remark}
If instead of having a stable periodic orbit, the unperturbed ODE has an unstable periodic orbit, then $\lambda_0$ in \eqref{invk} is positive. Analogous results to Theorems \ref{thm:zero} and  \ref{thm:all} will give us the parameterization of the periodic orbit and the unstable manifold for small $\ep$. The same proof, only with minor modifications, will work. At the same time, the invariance equation \eqref{inv} will be well-defined for a suitably chosen domain for $W$, we do not need to do extensions. Similarly, the idea here will also work for advanced equations, which have the same format as equation \eqref{SDDE}, with $r:\R^2\to[-h,0]$. We omit the details for these cases.
\end{remark}

\subsection{Representation of the unknown function}\label{formw}
In order to study the functional equation \eqref{invv}, we consider $W$ of the form:
\begin{equation}\label{wform}
W(\theta, s)=\sum^{N-1}_{j=0}W^j(\theta)s^j+W^>(\theta, s),
\end{equation} 
solving \eqref{invv}. Where $W^0(\theta)$ is the zeroth order term in $s$, $W^j(\theta)s^j$ is the j-th order term in $s$, $W^>(\theta, s)$ is of order at least $N$ in $s$. $W^j: \mathbb{T} \to \mathbb{T}\times \R$, and $W^>:\mathbb{T}\times \R\to \mathbb{T}\times \R$. From now on, we will use superscripts to denote corresponding orders, and subscripts, as we did before, to denote corresponding components.
 
We consider lifts of $W^0(\theta)$, $W^j(\theta)$, and $W^>(\theta, s)$, which will be functions from $\R\to \R^2$ or $\R^2\to \R^2$. We will not distinguish notations for the functions before or after lifts. According to the periodicity condition for $W$ in \eqref{pc}, the lifted functions satisfy the following periodicity conditions: 
\begin{align}\label{pc0}
&W^0(\theta+1)=W^0(\theta)+\left(\begin{smallmatrix}1\\0\end{smallmatrix}\right),\\
\label{pc1}
&W^j(\theta+1)=W^j(\theta),\\
\label{pc>}
&W^>(\theta+1,s)=W^>(\theta,s).
\end{align}

Based on the form of $W$ in \eqref{wform}, we can match coefficients of different powers of $s$ in the invariance equation \eqref{invv}. Thus, the invariance equation \eqref{invv} is equivalent to a sequence of equations. As we will see, the equations for $W^0$ and $W^1$ are special. The equation for $W^0$ is very nonlinear, the equation for $W^1$ is a relative eigenvector equation. The equations for $W^j$'s are all similar. The equation for $W^>$ is hard to study, it has 2 variables. As we will see later, for small enough $\ep$, $W^>$ is the only case where we need the cut-off.

\subsubsection{Invariance equation for zero order term}

Matching zero order terms of $s$ in \eqref{invv}, we obtain the equation for the unknowns $\omega$ and $W^0$:
\begin{equation}\label{inv0}
\omega \frac{d}{d{\theta}}W^0(\theta)-
\begin{pmatrix}
\omega_0\\ \lambda_0W^0_2(\theta)
\end{pmatrix}=\varepsilon\overline{Y}(W^0(\theta),\widetilde{W}^0(\theta; \omega), \varepsilon),
\end{equation}
where
\[
\widetilde{W}^0(\theta;\omega)=W^0\left(\theta-\omega\overline{r\circ K}(W^0(\theta))\right)
\]
is the function caused by delay.

\subsubsection{Invariance equation for first order term}

Equating the coefficients of $s^1$ in equation \eqref{invv}, we obtain the equation for the unknowns $\lambda$ and $W^1$:
\begin{equation}\label{inv1}
\omega\frac{d}{d\theta}W^1(\theta)+\lambda W^1(\theta)-\begin{pmatrix}
0\\ \lambda_0W^1_2(\theta)
\end{pmatrix}=\varepsilon \overline{Y}^1(\theta,\lambda, W^0, W^1,\varepsilon),
\end{equation}
where $\overline{Y}^1(\theta,\lambda,W^0, W^1,\varepsilon)$ is the coefficient of $s$ in $\overline{Y}$. $\overline{Y}^1(\theta,\lambda,W^0, W^1,\varepsilon)$ is linear in $W^1$. We will specify it later in \eqref{formy1}.

\subsubsection{Invariance equation for the j-th order term}

For $2\leq j\leq N-1$, matching the coefficients of $s^j$, the equation for the unknown $W^j$ is:
\begin{equation}\label{invj}
\omega\frac{d}{d\theta}W^j(\theta)+\lambda jW^j(\theta)-\begin{pmatrix}
0\\ \lambda_0W^j_2(\theta)
\end{pmatrix}=\varepsilon \overline{Y}^j(\theta,\lambda, W^0, W^j,\varepsilon)+R^j(\theta),
\end{equation}
where $\overline{Y}^j(\theta, W^0, W^j,\varepsilon)$ is the coefficient of $s^j$ in $\overline{Y}$. $\overline{Y}^j(\theta, W^0, W^j,\varepsilon)$ is linear in $W^j$. We will specify it later in \eqref{formyj}.
and $R^j$ is a function of $\theta$ which depends only on $W^0$, 
$W^1$,$\dotsc$, $W^{j-1}$.

Having $W^0, \dotsc, W^{N-1}$, we are ready to consider $W^>$. As we will see, the truncation number $N$ could be chosen as any integer larger than $1$ to obtain the main result of this paper.

\subsubsection{Invariance equation for higher order term}

For $W^>(\theta, s)$, it solves the equation:
\begin{equation}\label{invh}
(\omega\partial_{\theta}+s\lambda\partial_s)W^>(\theta,s)-\begin{pmatrix}
0\\\lambda_0W^>_2(\theta, s)
\end{pmatrix}=\varepsilon Y^>(W^>,\theta, s, \varepsilon)
\end{equation}
where $Y^>(W^>,\theta, s, \varepsilon)$ is the term of order at least $N$ in $s$ of $\overline{Y}$, which will be specified later in \eqref{formy>}.

\section{Some basic definitions and basic results on function spaces}\label{sc:basic}

In this section, we collect some standard results on the spaces of 
continuously differentiable functions that we will use.  
 
We will denote by $C^L(Y,X)$  the space of all functions from (an open subset of) a Banach space $Y$ to a Banach space $X$, with uniformly bounded continuous derivatives up to order $L$. We endow  $C^L(Y,X)$ with 
the norm
\[
\|f\|_{C^L}=\max_{0\leq j\leq L}\sup_{\xi\in Y}\|D^jf(\xi)\|_{Y^{\otimes j}\to X},
\]
so that $C^L(Y,X)$ is a Banach space. 

Note that we include in the  definition that the derivatives are uniformly bounded. 
This  is not the same as the Whitney topology on spaces of $L$ times differentiable 
functions in a $\sigma$-compact manifold \cite[p. 40]{GolubitskyG73}, which is a Fréchet 
topology. Even more general definitions appear in \cite{KrieglM97}.

We use $C^L_B(Y,X)$ to denote the closed subset of $C^L(Y,X)$ which consists of functions with $\|\cdot\|_{C^L}$ norm bounded by constant $B$.

We will also denote $C^{L+Lip}(Y,X)$ the space of $C^{L}$ functions with $L$-th derivative Lipschitz. We define
\[
\text{Lip}( D^Lf)=\sup_{\xi_1\neq\xi_2}\frac{\|D^Lf(\xi_1)-D^Lf(\xi_2)\|_{Y^{\otimes L}\to X}}{\|\xi_1-\xi_2\|_Y},
\]
and the norm $\|\cdot\|_{C^{L+Lip}(Y,X)}$ as the maximum of the $\|\cdot\|_{C^L}$ norm and $\text{Lip}(D^Lf)$. 

Define $C^{L+Lip}_B(Y,X)$ to be the closed subset of the space $C^{L+Lip}(Y,X)$ consisting of all functions of norm $\|\cdot\|_{C^{L+Lip}(Y,X)}$ bounded by the constant $B$. 

\subsection{Closure of $C^r$ balls in weak topology}

We quote proposition A2 in \cite{Lan}, as it will be used several times throughout this paper. It can be interpreted as $C^{L+Lip}_1(Y,X)$ is closed under pointwise weak topology on $X$. A related notion, Quasi-Banach space, was used in \cite{HartungTuri}.

\begin{lemma}[Lanford]\label{lem:Lan}
 Let $(u_n)_{n\in \N}$ be a sequence of functions on a Banach space $Y$ with values on a Banach space $X$. Assume that for all $n$, $y$
\[
\|D^ju_n(y)\|\leq 1\quad j=0, 1, 2, \dotsc, k,
\]
and that each $D^ku_n$  is Lipschitz with Lipschitz constant 1. Assume also that for each $y$, the sequence $(u_n(y))$ converges weakly(i.e., in the weak topology of $X$) to $u(y)$. Then,\\
(a) $u$ has a Lipschitz k-th derivative with Lipschitz constant 1;\\
(b) $D^ju_n(y)$ converges weakly to $D^ju(y)$ for all $y$ and $j=1,2,\dotsc,k.$
\end{lemma}

Note that the assumption of weak convergence of $(u_n(y))$, and part (b) in 
the conclusion implies that $\|D^j u(y)\|\leq 1$ for all $y$ and $j=0,1,2,\dotsc,k.$

As mentioned in \cite{Lan}, if $X$ and $Y$ are finite dimensional, the above lemma is just an application of Arzela-Ascoli Theorem. This is actually the only case we need. For the proof of above lemma in the general case, we refer to \cite{Lan}.

\subsection{Faà di Bruno formula}

We also quote Faà di Bruno formula, which deals with the derivatives of composition of two functions.

\begin{lemma}\label{faa}
Let $g(x)$ be defined on a neighborhood of $x^0$ in a Banach space $E$, and have derivatives up to order $n$ at $x^0$. Let $f(y)$ be defined on a neighborhood of $y^0=g(x^0)$ in a Banach space $F$, and have derivatives up to order $n$ at $y^0$. Then, the nth derivative of the composition $h(x)=f[g(x)]$ at $x^0$ is given by the formula
\begin{equation}
h_n=\sum^n_{k=1}f_k\sum_{p(n,k)}n!\prod^n_{i=1}\frac{g_i^{\lambda_i}}{(\lambda_i!)(i!)^{\lambda_i}}.
\end{equation}
In the above expression, we set
\[
h_n=\frac{d^n}{dx^n}h(x^0),\quad f_k=\frac{d^k}{dy^k}f(y^0),\quad g_i=\frac{d^i}{dx^i}g(x^0),
\]
and
\[
p(n,k)=\left\{(\lambda_1,\dotsc,\lambda_n):\lambda_i\in \N, 
\sum^n_{i=1}\lambda_i=k, \sum^n_{i=1}i\lambda_i=n\right\}.
\]
\end{lemma}
This can be proved by the Chain Rule and induction. See \cite{AbrahamR67} for a proof.

\subsection{Interpolation}
The interpolation inequalities will also be used many times. One can refer to \cite{Had98, Kol, LO99} for some related work. We quote the following result from \cite{LO99}:

\begin{lemma}\label{lem:inter}
Let $U$ be a convex and bounded open subset of a Banach space $ E$, $F$ be a Banach space. Let $r$, $s$, $t$ be positive numbers, $0\leq r<s<t$, and $\mu=\frac{t-s}{t-r}$. There is a constant $M_{r,t}$, such that if $f\in C^t(U,F)$, then 
\[
\|f\|_{C^s}\leq M_{r,t}\|f\|_{C^r}^{\mu}\|f\|_{C^t}^{1-\mu}.
\]
\end{lemma}

\section{Main results} \label{sc:mainr}
\subsection{Results for prepared equations}\label{ssc:prep}

Under the assumption that the map $\overline{Y}:({\T}\times\R)^2\times \R_+\to\R^2$ has bounded derivatives up to any order, $\overline{r\circ K}:{\T}\times\R\to [0,h]$ has bounded derivatives up to any order, we have:
\begin{theorem}[Zero Order]\label{thm:zero}
For any given integer $L>0$, there is $\ep_0>0$ such that when $0\leq\varepsilon< \ep_0$, there exist an $\omega>0$ and an $L$ times differentiable map $W^0: \mathbb{T} \to \mathbb{T}\times \R$, with $L$-th derivative Lipschitz, which solve equation \eqref{inv0}.

Moreover, for initial guess $\omega^0$, and $W^{0,0}(\theta)$ satisfying the periodicity condition \eqref{pc0}. If they satisfy the invariance equation \eqref{inv0} with error $E^0(\theta)$, then there exist unique $\omega$, $W^0(\theta)$(satisfying the periodic condition \eqref{pc0}) closed by solving the same equation exactly, with 
\begin{align}\label{aprw0}
\|W^{0,0}-W^0\|_{C^l}\leq& C\|E^0\|^{1-\frac{l}{L}}_{C^0},\quad0\leq l<L\\
\label{aprom}
|\omega^0-\omega|\leq& C\|E^0\|_{C^0},
\end{align}
for some constant $C$, where $C$ may depend on $\varepsilon$, $\omega_0$, $\lambda_0$, $l$, $L$, and prior bound for $\|W^{0,0}\|_{L+Lip}$. In fact, $W^0$ has derivatives up to any order. 
\end{theorem}

Moreover,
\begin{theorem}[All Orders]\label{thm:all}
For any given integers $N\geq 2$, and $L\geq 2+N$, there is $\ep_0>0$ such that when $0\leq\varepsilon < \ep_0$, there exist $\omega> 0$, $\lambda<0$, and $W: \mathbb{T}\times\mathbb{R}\to\mathbb{T}\times\mathbb{R}$ of the form 
\begin{equation}\label{wform1}
W(\theta,s)=\sum_{j=0}^{N-1}W^j(\theta)s^j+W^>(\theta,s)
\end{equation}
which solve the equation \eqref{invv} in a neighborhood of $s=0$.

Where $W^0: \mathbb{T} \to\mathbb{T}\times\mathbb{R}$ is $L$ times differentiable with Lipschitz $L$-th derivative. For $1\leq j\leq N-1$, $W^j:  \mathbb{T} \to \mathbb{T}\times\mathbb{R}$ is $(L-1)$ times differentiable with Lipschitz $(L-1)$-th derivative, and $W^>$ is of order at least N in $s$ and is jointly $(L-2-N)$ times differentiable in $\theta$ and $s$, with $(L-2-N)$-th derivative Lipschitz.

Moreover, if $\omega^0$, $W^{0,0}(\theta)$, $\lambda^0$, $W^{1,0}(\theta)$, $W^{j,0}(\theta)$, and $W^{>,0}(\theta,s)$ satisfy the invariance equations \eqref{inv0}, \eqref{inv1}, \eqref{invj}, and \eqref{invh}, with errors $E^0(\theta)$, $E^1(\theta)$, $E^j(\theta)$, and $E^>(\theta,s)$, respectively, then there are $\omega$, $W^{0}(\theta)$, $\lambda$, $W^{1}(\theta)$, $W^{1}(\theta)$, and $W^{>}(\theta,s)$ which solve equations \eqref{inv0}, \eqref{inv1}, \eqref{invj}, and \eqref{invh}. Therefore, equation \eqref{invv} is solved by $\omega$, $\lambda$, and $W(\theta, s)$ of above form \eqref{wform1}. For $0\leq l\leq L-2-N$, we have
\begin{equation}\label{aprw}
\begin{split}
\|W(\theta, s)&-\sum_{j=0}^{N-1}W^{j,0}(\theta)s^j-W^{>,0}(\theta, s)\|_{C^l}\\
&\leq C(\sum_{j=0}^{N-1}\|E^j\|_{C^0}|s|^j+\|E^>\|_{0,N}|s|^N)^{1-\frac{l}{(L-2-N)}}, 
\end{split}
\end{equation}
$$|\omega-\omega^0|\leq C(\|E^0\|_{C^0}),$$
\begin{equation}\label{aprlam}
|\lambda-\lambda^0|\leq C(\|E^1\|_{C^0}),
\end{equation}
for some constant $C$ depending on $\varepsilon$, $\omega_0$, $\lambda_0$, $N$, $l$, $L$, prior bounds for $\|W^{0,0}\|_{L+Lip}$, $\|W^{j,0}\|_{L-1+Lip}$, $j=1,\dotsc, N-1$, and derivatives of $W^{>,0}$.
\end{theorem}

\begin{remark}
In Theorem \ref{thm:zero}, $W^0(\theta)$ is unique up to a phase shift.
\end{remark}

\begin{remark}
The above Theorems are in a-posteriori format. The main input needed is 
some function that satisfies the invariance equation approximately. This can be numerical computations (that indeed produce good approximate solutions) or Lindstedt series, see for example \cite{Livia19}. 

Notice that with these 
Theorems, we do not need to analyze the procedure used to produce the approximate solutions. 
The only thing that we need to establish is that the solutions produced satisfy the invariance
equation up to a small error. 
\end{remark}

The a-posteriori format of the  theorem leads to automatic Hölder dependence of the solutions $W^0$ on $\ep$ and $Y$.

It suffices to observe that if we consider $W^0$ solving the invariance equation for some 
$\ep_1, Y_1$, it will solve the invariance equation for $\ep_2, Y_2$ up to an error 
which is bounded in the $C^l$ norm by $C\left(| \ep_1 - \ep_2| + \|Y_1  - Y_2\|_{C^0}\right)^{1-\frac{l}{L}}$

As a matter of fact, one of the advantages of our approach is that it leads 
very easily to smooth dependence on parameters.

\begin{theorem}\label{thm:smooth} 
Consider a family of functions $Y_\eta, r_\eta$ as above, where $\eta$ lies in an open interval $I\subset\R$. 
Assume that $Y_{\eta}$ and $r_{\eta}$ are smooth in their inputs as well as in $\eta$, with bounded derivatives. 

Then for any positive integer $L$,  there is an $\ep_0$ small enough such that when $\ep<\ep_0$, for each $\eta \in I$ we can find 
$\omega_\eta$, $W^0_\eta$ solving \eqref{inv0}. 

Furthermore, the $W^0_\eta(\theta)$ is jointly $C^{L+ Lip}$ in $\eta$, $\theta$. 
\end{theorem} 

\begin{theorem}\label{thm:smoothh}
Under assumption of Theorem \ref{thm:smooth}, for any given integers $N\geq2$, and $L\geq2+N$, there is an $\ep_0$ small enough such that when $\ep<\ep_0$, for each $\eta \in I$, we can find $\omega_\eta$, $W^0_\eta$, $\lambda_\eta$, $W^j_\eta$, $j=1,\dotsc,N-1$, and $W^>_\eta(\theta,s)$, which solve the invariance equations \eqref{inv0}, \eqref{inv1}, \eqref{invj}, and \eqref{invh}.

Furthermore, $W^0_\eta(\theta)$ is jointly $C^{L+ Lip}$ in $\eta$, $\theta$; $W^j_\eta(\theta)$, $j=1,\dotsc,N-1$, are jointly $C^{L-1+ Lip}$ in $\eta$, $\theta$; $W^>_\eta(\theta,s)$ is jointly $C^{L-2-N+ Lip}$ in $\eta$, $\theta$, and $s$.
\end{theorem}

Note that the regularity conclusions of Theorem~\ref{thm:smooth} can be 
interpreted in a more functional form as saying that the 
mapping that to $\eta$ associates $W^0_\eta$ is $C^{\ell + \text{Lip}} $
when the space of embedding $W$ is given the $C^{L - \ell}$ 
topology. Similar interpretation can be made for Theorem \ref{thm:smoothh}. This functional point of view is consistent with 
the point of view of RFDE where the phase space is 
infinite dimensional. 

\subsection{Results for original problem in a neighborhood of the limit cycle}\label{ssc:orig}
Note that to find the low order terms, $W^0,\dotsc,W^j$, for small $\ep$, the extensions are not needed.
Heuristically, the low order terms are \emph{infinitesimals}. Hence, to compute them, it suffices to 
know the expansion of the vector field.

More precisely, if we take the initial guess for zero order term as $W^{0,0}(\theta)= \left(\begin{smallmatrix}\theta\\0\end{smallmatrix}\right)$, the error for this initial guess is of order $\ep$. Then by theorem \ref{thm:zero}, the true solution $W^0$ is within a distance of order $\varepsilon$ from $W^{0,0}(\theta)$. Therefore, with small choice of $\ep$, we can have $\sup_{\theta\in\mathbb{T}}|W^0_2(\theta)|<\frac{1}{2}$, we are reduced to the case without extension, since
\begin{align*}
&\overline{r\circ K}(W^0(\theta))={r\circ K}(W^0(\theta)),\\
&\overline{Y}(W^0(\theta),\widetilde{W}^0(\theta;\omega), \varepsilon)=Y(W^0(\theta),\widetilde{W}^0(\theta;\omega), \varepsilon).
\end{align*}
where,
\[
\widetilde{W}^0(\theta;\omega)=W^0(\theta-\omega r\circ K(W^0(\theta))).
\]
Then we can rewrite the invariance equation for $W^0$, \eqref{inv0}, as:
\begin{equation}\label{invv0}
\omega \frac{d}{d{\theta}}W^0(\theta)-
\begin{pmatrix}
\omega_0\\ \lambda_0W^0_2(\theta)
\end{pmatrix}=\varepsilon Y(W^0(\theta),\widetilde{W}^0(\theta; \omega), \varepsilon).
\end{equation}
 
 Similar arguments apply for the equations for $W^1$ and $W^j$'s if
we look at expressions of $\overline{Y}^1$ in \eqref{formy1},
$\overline{Y}^j$ in \eqref{formyj}, and form of $R^j$.

We can find $0<s_0<\frac{1}{2}$, such that
$W(\mathbb{T}\times[-s_0,s_0]) \subset \mathbb{T}\times
[-\frac{1}{2},\frac{1}{2}]$, and
$\widetilde{W}(\mathbb{T}\times[-s_0,s_0]) \subset \mathbb{T}\times
[-\frac{1}{2},\frac{1}{2}]$. Therefore, the original problem is solved
in a neighborhood of the limit cycle by applying the results in
section \ref{ssc:prep}.

For the original problem in section \ref{sc:formu}, we have
\begin{cor}[Limit Cycle]\label{Cor:lc}
When $\ep< \ep_0$ in Theorem \ref{thm:zero} is so small that $\sup_{\theta\in\mathbb{T}}|W^0_2(\theta)|<\frac{1}{2}$, equation \eqref{SDDE} admits a limit cycle close to the limit cycle of the unperturbed equation. If $\omega$, $W^0$ solve the invariance equation \eqref{invv0}, then $K\circ W^0(\theta)$ gives a parameterization of the limit cycle of equation \eqref{SDDE}, i.e. for any $\theta$, $K\circ W^0(\theta+\omega t)$ solves equation \eqref{SDDE} for all $t$.
\end{cor}

We can also find a 2-parameter family of solutions close to the limit cycle:
\begin{cor}[Isochrons]
For small $\ep$ as in previous Corollary \ref{Cor:lc}, there are isochrons for the limit cycle of equation \eqref{SDDE}. If $\omega$, $\lambda$, and $W:\mathbb{T}\times\R\to\mathbb{T}\times\R$ solve the extended invariance equation \eqref{invv}, then there exists $0<s_0<\frac{1}{2}$, such that $K\circ W(\theta,s)$, $|s|\leq s_0$ gives a parameterization of the limit cycle with its isochrons in a neighborhood, i.e. for any $\theta$, and $s$, with $|s|\leq s_0$, $K\circ W(\theta+\omega t,se^{\lambda t})$ solves equation \eqref{SDDE} for all $t\geq0$.
\end{cor}

Dependence on parameters results  in Theorem \ref{thm:smooth} and \ref{thm:smoothh} apply.


\subsection{Comparison with Results on RFDE based on time evolution}
\label{ssc:evolution} 

The persistence of a periodic solution under perturbation for retarded functional 
differential equation (RFDE) is presented 
in Chapter 10 of \cite{HaleLunel}, notably Theorem $4.1$.  In this section, we present some remarks that can help the specialists to compare our results with 
those obtainable considering the time evolution of RFDEs. 

The set up presented there does not seem to apply to our case since the phase space considered in \cite{HaleLunel} is the space of continuous functions on an interval, namely, $C^0[-h,0]$, and they require differentiability properties of the equation which are not satisfied in our case. Note also that we can obtain smooth dependence on parameters (see Theorem~\ref{thm:smooth}). 
Obtaining such smooth dependence using the methods based on the evolutionary approach would require obtaining regularity of the evolution operator, which does not seem to be available. 

More precisely, if we employ the notation $x_t$ as a function defined on $[-h,0]$, with 
\[
x_t(s)=x(t+s)
\] for $s\in[-h,0]$, we can write our SDDE \eqref{SDDE} as
\[
\dot x(t)=F(x_t,\ep),
\]
where we define $F(\phi,\ep):=X(\phi(0),\ep\phi(-r(\phi(0))))$. For $\ep=0$, we have an ODE, which can be viewed as a delay equation, with a non-degenerate periodic orbit (see \cite{HaleLunel}). However, above $F$ cannot be continuously differentiable in $\phi$ if $\phi$ is only continuous. This obstructs application of Theorem $4.1$ for RFDE in \cite{HaleLunel}. 

It is very interesting to study whether a similar method to the one in \cite{HaleLunel} can be extended to our case with some variations of the phase space (solution manifold, see \cite{Walt}). However, since only $C^1$ regularity of the evolution has been proved(\cite{Walt}), (higher regularity of the evolution in SDDE seems problematic), one cannot hope to obtain the dependence on parameters to be more  regular than $C^1$. On the other hand, the method  in this paper allows to get rather straightforwardly higher smoothness with respect to parameters. 
See Theorem~\ref{thm:smooth}. 
We mention that some progress in continuation of periodic orbits is in \cite{MPNP}.

Considering RFDE's as evolutions in infinite dimensional phase spaces, \cite{HaleLunel} establishes the existence of infinite-dimensional strong stable manifolds for periodic orbits corresponding to the Floquet multipliers smaller than a number. 

Again, we remark that there are some technical issues of regularity of evolutions in phase space of SDDE to define stable manifolds and even stability. We hope that these regularity issues of the evolution can be made precise (using techniques as in \cite{Walt,Obaya, MPN}).

Nevertheless, there is a very fundamental difference between the manifolds we consider and those in \cite{HaleLunel}.

If we consider the unperturbed ODE as an RFDE in an infinite dimensional phase space, the Floquet multipliers are $1$ with multiplicity 1, $\exp{(\frac{\lambda_0}{\omega_0})}$ with multiplicity 1, and $0$ (with infinite multiplicity). With $C^1$-smoothness of the evolution as in \cite{Walt}, under small perturbation, we would have the Floquet multipliers be similar to those (one exactly $1$, one close to $\exp({\frac{\lambda_0}{\omega_0}})$ and infinitely many near $0$).

The theory developed in \cite{HaleLunel} attaches an infinite-dimensional manifold to the most stable part of the spectrum. That is the strong stable manifold.

The manifold that we consider here, in the infinite-dimensional phase space, is attached to the least stable Floquet multiplier, hence it is a slow stable manifold from the infinite-dimensional point of view. 

We think that the finite-dimensional manifold we obtain are more practically relevant
than the strong stable manifold. We expect that infinitely many modes
will die out the fastest and, therefore, be hard to observe. All the solutions of the full problem will 
be asymptotically similar to the solutions we consider.  
In summary, solutions close to the limit cycle will converge to the limit cycle tangent to the slow stable manifolds
described here. One problem to make all this precise is that the evolution is 
only known to be $C^1$. 

Our motivation is to obtain solutions which resemble solutions of the ODE, in accordance with the physical intuition that the solutions in the perturbed problem -- in spite of the singular nature of
the perturbation -- look similar to those of the unperturbed problem (this is  the reason why relativity
and its delays were hard to discover). 
 
One of the features of the formalism in this paper is that it allows to describe in a unified way the solutions 
of the SDDE in an infinite dimensional space and the finite dimensional solutions of the unperturbed ODE problem.

Of course in this paper, we only deal with models of a very special kind (we indeed have 
the hope that the range of applicability of the method can be expanded; the models considered 
in this paper are a proof of concept) 
but we obtain rather smooth invariant manifolds and
smooth dependence on parameters with high degree of differentiability.
Furthermore, the proof presented here  leads to algorithms to compute the limit cycles and 
their manifolds. These algorithms are practical and have been implemented, see \cite{Joan19}.

It is also interesting to investigate whether evolution based methods lead to computational algorithms \cite{JoanTh} and compare them with the algorithms based on functional equations as in \cite{Joan19}.

\section{Overview of the proof} \label{sc:overpf}
In equation \eqref{inv0}, $\omega$ and $W^0$ are unknowns. Under a choice of the phase, we define an operator such that its fixed point solves \eqref{inv0}. We will show that when $\ep$ is small enough, the operator is a ``$C^0$" contraction and maps a $C^{L+Lip}$ ball to itself. Then one obtains the existence of the fixed point $(\omega, W^0)$, and that $W^0$ in the fixed point has some regularity. Therefore, equation \eqref{inv0} is solved.

In equation \eqref{inv1}, $\lambda$ and $W^1$ are unknowns. We will impose an appropriate normalization when defining the operator to make sure the solution is uniquely found, and that $W$ is close to the identity map with appropriate scaling factor. Then similar to above case, for small enough $\varepsilon$, this operator has a fixed point $(\lambda, W^1)$ in which $W^1$ has some regularity.

In equation \eqref{invj}, $W^j$ is the only unknown. We define an operator which is a contraction for small enough $\ep$. The operator has a fixed point with certain regularity solving the equation.

In equation \eqref{invh}, $W^>$ is an unknown function of 2 variables. We will define an operator on a function space with a weighted norm, then prove for small $\varepsilon$, this operator has a fixed point in this function space, which solves the equation \eqref{invh}.

We emphasis again that for small enough $\ep$, the equation for $W^>$ is the only place where extension is needed. (Recall section \ref{ssc:orig})

There are finitely many smallness conditions for $\ep$,  so there are $\ep$'s which satisfy all the smallness conditions.

Same idea will be used for proving the smooth dependence on parameters.

\section{Proof of the main results} \label{sc:pf}
\subsection{Zero Order Solution}
\label{sc:plc}
In this section, we prove our first result, Theorem \ref{thm:zero}.

Recall \eqref{inv0}, invariance equation for $\omega$ and $W^0$, as in section \ref{formw}, which is obtained by setting $s=0$ in equation \eqref{invv}.

Componentwise, $W^0=(W^0_1, W^0_2)$, and $\overline{Y}=(\overline{Y}_1, \overline{Y}_2)$, we have the equations as:
\begin{align}
&\omega \frac{d}{d{\theta}}W^0_1(\theta)-
\omega_0 =\varepsilon \overline{Y}_1(W^0(\theta),\widetilde{W}^0(\theta;\omega), 
\varepsilon),\label{inv01} 
\intertext{and}
&\omega \frac{d}{d{\theta}}W^0_2(\theta)- \lambda_0W^0_2(\theta) =\varepsilon 
\overline{Y}_2(W^0(\theta),\widetilde{W}^0(\theta;\omega), \varepsilon). \label{inv02}
\end{align}

Taking periodicity condition \eqref{pc0} into account, we define an operator $\Gamma^0$ as follows:
\begin{equation}\label{op0}
\begin{split}
\Gamma^0\begin{pmatrix}
{a}\\
{Z}_1 \\
{Z}_2
\end{pmatrix}(\theta)&=
\begin{pmatrix}
\Gamma^0_1({a},{Z})\\
\Gamma^0_2({a},{Z})(\theta)\\
\Gamma^0_3({a},{Z})(\theta)
\end{pmatrix}\\&=
\begin{pmatrix}
\omega_0+\varepsilon \int^1_0 \overline{Y}_1({Z}(\theta),\widetilde{{Z}}(\theta;{a}), \varepsilon)d\theta \\
\frac{1}{\Gamma^0_1({a},~{Z})}\big(\omega_0\theta+\varepsilon \int^{\theta}_0 \overline{Y}_1({Z}(\sigma),\widetilde{{Z}}(\sigma;{a}), \varepsilon) d\sigma\big) \\
\varepsilon\int_0^{\infty}e^{\lambda_0 t} \overline{Y}_2({Z}(\theta-{a} t),\widetilde{{Z}}(\theta-{a} t;{a}), \varepsilon) dt
\end{pmatrix},
\end{split}
\end{equation}

Notice that if $\Gamma^0$ has a fixed point $({a}^*,{Z}^*)$, then \eqref{inv0} are solved by ${a}^*$ and ${Z}^*$, at the same time, periodic condition \eqref{pc0} is satisfied.\\
\begin{remark}
As we can see, the operator $\Gamma^0$ will depend on $\varepsilon$, however, to simplify the expression, we will not include $\varepsilon$ in the notation of the operator $\Gamma^0$.
\end{remark}

\begin{remark}
Similar to Remark \ref{rmk:nonuni}, we will not have uniqueness of the solution to invariance equation \eqref{inv0}. Once we have a solution $W^0(\theta)$ to the equation, for any $\theta_0$, $W^0(\theta+\theta_0)$ will also solve the equation, which is called phase shift. This is indeed the only source of non-uniqueness.

By considering the operator \eqref{op0}, we fix a phase by $\Gamma^0_2({a},{Z})(0)=0$.
\end{remark}

For the domain of $\Gamma^0$, we consider a closed interval $I^0=\{{a}: |{a}-\omega_0|\leq\frac{\omega_0}{2}\}$. For a fixed positive integer $L$, define a subset of the space of functions which are $L$ times differentiable, with Lipschitz $L$-th derivative as follows:
\begin{align}
\C^{L+Lip}_0=\{f\ |\ f:&{\T}\to{\T}\times\R, f\text{ can be lifted to a function from } \R \text{ to } \R^2, \nonumber\\ 
&\text{still denoted as }f, \text{which satisfies } f(\theta+1)=f(\theta)+\left(\begin{smallmatrix}1\\0\end{smallmatrix}\right),\nonumber \\
&f_1(0)=0, \|f\|_{L+Lip}\leq B^0\},
\end{align}

where 
\[
\|f\|_{L+Lip}=\max_{i=1,2, k=0,\dotsc,L}\{\sup_{\theta\in[0,1]}\|f^{(k)}_i(\theta)\|,~Lip(f^{(L)}_i)\}.
\]

Define $D^0=I^0\times \C^{L+Lip}_0$, then $\Gamma^0$ is defined on $D^0$. We have the following:

\begin{lemma}\label{lem:prob0}
There exists $\ep^0>0$, such that when $\varepsilon<\ep^0$, $\Gamma^0(D^0)\subset D^0$. 
\end{lemma}
\begin{proof} 
For $({a},{Z}) \in D^0$, by assumption, we have that $\overline{Y}_1({Z}(\theta),\widetilde{{Z}}(\theta;{a}), \varepsilon)$ is bounded by a constant which is independent of choice of $({a},{Z})$ in $D^0$. Then, one can choose $\varepsilon$ small enough such that $\Gamma^0_1({a},{Z})=\omega_0+\varepsilon \int^1_0 \overline{Y}_1({Z}(\theta),\widetilde{{Z}}(\theta;{a}), \varepsilon)d\theta$ is in $I^0$.

Now consider $\Gamma^0_2({a},{Z})(\theta)=\frac{1}{\Gamma^0_1({a},{Z})}\big(\omega_0\theta+\varepsilon \int^{\theta}_0 \overline{Y}_1({Z}(\sigma),\widetilde{{Z}}(\sigma;{a}), \varepsilon) d\sigma\big)  $. First we observe that
\[
\Gamma^0_2({a},{Z})(\theta+1)=\Gamma^0_2({a},{Z})(\theta)+1.
\] 
Then we need to check bounds for the derivatives
\[
\frac{d}{d\theta}\Gamma^0_2({a},{Z})(\theta)=\frac{1}{\Gamma^0_1({a},{Z})}\big(\omega_0+\varepsilon \overline{Y}_1({Z}(\theta),\widetilde{{Z}}(\theta;{a}), \varepsilon)\big).
\]
By Fa\'a  di Bruno's formula in Lemma \ref{faa}, for $2\leq n\leq L$, $\frac{d^n}{d\theta^n}\Gamma^0_2({a},{Z})(\theta)$ will be a polynomial of a common factor $\frac{\varepsilon}{\Gamma^0_1({a},{Z})}$, each term will contain products of derivatives of $\overline{Y}_1$, ${Z}$, and $\overline{r\circ K}$ up to order $(n-1)$. By assumption on $\overline{Y}_1$ and $\overline{r\circ K}$, for $({a},{Z}) \in D^0$, if we choose $B^0$ to be larger than 2, then for small enough $\varepsilon$,   $\Gamma^0_2({a},{Z})(\theta)$ on $[0,1]$ has derivatives up to order $L$ bounded by $B^0$ and $L-th$ derivative Lipschitz with Lipschitz constant less than $B^0$.

For $\Gamma^0_3({a},{Z})(\theta)=\varepsilon\int_0^{\infty}e^{\lambda_0 t} \overline{Y}_2({Z}(\theta-{a} t),\widetilde{{Z}}(\theta-{a} t;{a}), \varepsilon) dt$. It satisfies
\[
\Gamma^0_3({a},{Z})(\theta+1)=\Gamma^0_3({a},{Z})(\theta).
\]
To establish bounds for the derivatives of $\Gamma^0_3({a},{Z})(\theta)$, we apply a similar argument as above. Notice that for $n\leq L$, $\frac{\partial^n}{\partial\theta^n}\overline{Y}_2({Z}(\theta-{a} t),\widetilde{{Z}}(\theta-{a} t;{a}), \varepsilon)$ will be a polynomial with each term a product of derivatives of $\overline{Y}_2$, ${Z}$, and $\overline{r\circ K}$ up to order $n$. With regularity of $\overline{Y}_2$, and $\overline{r\circ K}$, for $({a},{Z}) \in D^0$, $|\frac{\partial^n}{\partial\theta^n}\overline{Y}_2({Z}(\theta-{a} t),\widetilde{{Z}}(\theta-{a} t), \varepsilon)|$ will be bounded. Therefore, for small enough $\varepsilon$, $\Gamma^0_3(a,{Z})$ has derivatives up to order $L$ bounded by $B^0$ and its $L-th$ derivative is Lipschitz with Lipschitz constant less than $B^0$.

If we take $\varepsilon^0$ such that above conditions are satisfied at the same time, then for $\varepsilon<\varepsilon^0$, we have $\Gamma^0(D^0)\subset D^0$.
\end{proof}

We now define a distance on $D^0$, which is essentially $C^0$ distance. Under this distance, the space $D^0$ is complete. For $({a}, {Z})$ and $({a'}, {Z'})$ in $D^0$, 
\begin{equation}\label{dist}
d(({a}, {Z}), ({a'}, {Z'}))=|{a}-{a'}|+\|Z-Z'\|,
\end{equation}
where 
\begin{equation}\label{norm}
\|Z-Z'\|=\max\left\{\sup_{\theta}|{Z}_1(\theta)-{Z'_1}(\theta)|, \sup_{\theta}|{Z}_2(\theta)-{Z'_2}(\theta)|\right\}.
\end{equation}
\begin{lemma}\label{lem:contraction0}
There exists $\ep^0>0$, such that when $\varepsilon<\ep^0$, under above choice of distance \eqref{dist} on $D^0$, the operator $\Gamma^0$ is a contraction. 
\end{lemma}

\begin{proof}

We will show that for $\varepsilon$ small enough,(the explicit  form 
of smallness conditions will become clear along the proof), we can find a constant $\mu_0<1$ such that for distance defined above in \eqref{dist}
\begin{equation}\label{ctr0}
d(\Gamma^0({a}, {Z}), \Gamma^0({a'}, {Z'}))<\mu_0 d(({a}, {Z}), ({a'}, {Z'})).
\end{equation}

Note that 
\begin{equation}\label{dis0}
\begin{split}
d(\Gamma^0({a}, {Z}), \Gamma^0({a'}, {Z'}))=
&\left|\Gamma^0_1({a}, {Z})-\Gamma^0_1({a'}, {Z'})\right|\\
&+\|(\Gamma^0_2({a}, {Z}),\Gamma^0_3({a}, {Z}))-(\Gamma^0_2({a'}, {Z'}),\Gamma^0_3({a'}, {Z'}))\|
\end{split}
\end{equation}

More explicitly, above distance is
\begin{equation}\label{dis00}
\begin{split}
\varepsilon &\left|\int^1_0 \overline{Y}_1({Z}(\theta),\widetilde{{Z}}(\theta;{a}), \varepsilon)d\theta-\int^1_0 \overline{Y}_1({Z'}(\theta),\widetilde{{Z'}}(\theta;{a'}), \varepsilon)d\theta\right|\\
&+\max\biggl\{\sup_{\theta}\biggl\lvert\frac{1}{\Gamma^0_1({a},{Z})}\big(\omega_0\theta+\varepsilon \int^{\theta}_0 \overline{Y}_1({Z}(\sigma),\widetilde{{Z}}(\sigma;{a}), \varepsilon) d\sigma\big)\\
&\phantom{AAAAAAAA}-\frac{1}{\Gamma^0_1({a'},{Z'})}\big(\omega_0\theta+\varepsilon \int^{\theta}_0 \overline{Y}_1({Z'}(\sigma),\widetilde{{Z'}}(\sigma;{a'}), \varepsilon) d\sigma\big)\biggr\rvert,\\
&\phantom{AAAAAAA}\varepsilon\sup_{\theta}\biggl\lvert\int_0^{\infty}e^{\lambda_0 t} \overline{Y}_2({Z}(\theta-{a} t),\widetilde{{Z}}(\theta-{a} t;{a}), \varepsilon) dt\\
&\phantom{AAAAAAAAAAA}-\int_0^{\infty}e^{\lambda_0 t} \overline{Y}_2({Z'}(\theta-{a'} t),\widetilde{{Z'}}(\theta-{a'} t;{a'}), \varepsilon) dt\biggr\rvert\biggr\}
\end{split}
\end{equation}

Now we consider each term of above expression \eqref{dis00}. Note that in the above expression, it suffices to take the supremums for $\theta\in[0,1]$, which follows from periodicity condition \eqref{pc0}. By adding and subtracting terms, we have
\begin{equation*}
\begin{split}
&\left|\overline{Y}_1({Z}(\theta),\widetilde{{Z}}(\theta;{a}), \varepsilon)- \overline{Y}_1({Z'}(\theta),\widetilde{{Z'}}(\theta;{a'}), \varepsilon)\right|\\
&=\left|\overline{Y}_1({Z}(\theta),{Z}(\theta-{a} \overline{r\circ K}({Z}(\theta))), \varepsilon)- \overline{Y}_1({Z'}(\theta),{Z'}(\theta-{a'} \overline{r\circ K}({Z'}(\theta))), \varepsilon)\right|\\
&\leq\left|\overline{Y}_1({Z}(\theta),{Z}(\theta-{a} \overline{r\circ K}({Z}(\theta))), \varepsilon)- \overline{Y}_1({Z'}(\theta),{Z}(\theta-{a} \overline{r\circ K}({Z}(\theta))), \varepsilon)\right|\\
&\phantom{A}+\left|\overline{Y}_1({Z'}(\theta),{Z}(\theta-{a} \overline{r\circ K}({Z}(\theta))), \varepsilon)-\overline{Y}_1({Z'}(\theta),{Z'}(\theta-{a} \overline{r\circ K}({Z}(\theta))), \varepsilon)\right|\\
&\phantom{A}+\left|\overline{Y}_1({Z'}(\theta),{Z'}(\theta-{a} \overline{r\circ K}({Z}(\theta))), \varepsilon)-\overline{Y}_1({Z'}(\theta),{Z'}(\theta-{a'} \overline{r\circ K}({Z}(\theta))), \varepsilon)\right|\\
&\phantom{A}+\left|\overline{Y}_1({Z'}(\theta),{Z'}(\theta-{a'}\overline{r\circ K}({Z}(\theta))), \varepsilon)-\overline{Y}_1({Z'}(\theta),{Z'}(\theta-{a'} \overline{r\circ K}({Z'}(\theta))), \varepsilon)\right|.
\end{split}
\end{equation*}

By the mean value theorem, and the fact that $({a}, {Z})$ and $({a'}, {Z'})$ are in $D^0$, we have
\begin{equation}\label{es0}
\begin{split}
\bigl|\overline{Y}_1({Z}(\theta),&\widetilde{{Z}}(\theta;{a}), \varepsilon)- \overline{Y}_1({Z'}(\theta),\widetilde{{Z'}}(\theta;{a'}), \varepsilon)\bigr|\\
&\leq 2\|D\overline{Y}_1\|\|{Z}-{Z'}\|+\|D\overline{Y}_1\| \|D{Z'}\| \|\overline{r\circ K}\||{a}-{a'}|\\
&\phantom{AA}+|D\overline{Y}_1\|\|D{Z'}\| |{a'}|\|D(\overline{r\circ K})\|\|{Z}-{Z'}\|\\
&\leq \|D\overline{Y}_1\|\left(2+B^0 |{a'}|\|D(\overline{r\circ K})\|\right)\|{Z}-{Z'}\|\\
&\phantom{AA}+\|D\overline{Y}_1\|B^0 \|\overline{r\circ K}\||{a}-{a'}|.
\end{split}
\end{equation}
Where all the norms are the usual supremum norms on $\R$ or $\R^2$ (defined as above in \eqref{norm}), with 
\begin{equation}\label{dynorm}
\|D\overline{Y}_1\|=\max\{\|D_1\overline{Y}_1\|, \|D_2\overline{Y}_1\|\},
\end{equation}
where $\|D_i\overline{Y}_1\|$, $i=1,2$ is the supremum of the operator norm corresponding to the infinity norm defined on $\R^2$.

Then, 
\begin{equation}\label{R1}
\begin{split}
\left|\Gamma^0_1({a}, {Z})-\Gamma^0_1({a'}, {Z'})\right|&\leq
\varepsilon\|D\overline{Y}_1\|\left(2+B^0 |{a'}|\|D(\overline{r\circ K})\|\right)\|Z-Z'\|\\
&\phantom{AA}+\varepsilon B^0\|D\overline{Y}_1\|\|\overline{r\circ K}\||{a}-{a'}|
\end{split}
\end{equation}

Now consider the first component of the maximum, for $\theta\in[0,1]$, by adding and subtracting terms, we have:
\begin{equation}\label{R20}
\begin{split}
&\left|\Gamma^0_2({a}, {Z})-\Gamma^0_2({a'}, {Z'})\right|\\ &\phantom{AA}\leq \frac{\varepsilon}{|\Gamma^0_1({a}, {Z})|}\int_0^1\left|\overline{Y}_1({Z}(\theta),\widetilde{{Z}}(\theta), \varepsilon)d\theta- \overline{Y}_1({Z'}(\theta),\widetilde{{Z'}}(\theta), \varepsilon)\right|d\theta\\
&\phantom{AAAA}+\frac{\varepsilon\int^1_0 \left|\overline{Y}_1({Z'}(\theta),\widetilde{{Z'}}(\theta;{a'}), \varepsilon)\right|d\theta}{|\Gamma^0_1({a}, {Z})\Gamma^0_1({a'}, {Z'})|}\left|\Gamma^0_1({a}, {Z})-\Gamma^0_1({a'}, {Z'})\right|\\
&\phantom{AAAA}+\frac{|\omega_0|}{|\Gamma^0_1({a}, {Z})\Gamma^0_1({a'}, {Z'})|}\left|\Gamma^0_1({a}, {Z})-\Gamma^0_1({a'}, {Z'})\right|\\
&\phantom{AA}\leq\frac{\varepsilon}{|\Gamma^0_1({a}, {Z})|}\int_0^1\left|\overline{Y}_1({Z}(\theta),\widetilde{{Z}}(\theta), \varepsilon)d\theta- \overline{Y}_1({Z'}(\theta),\widetilde{{Z'}}(\theta), \varepsilon)\right|d\theta\\
&\phantom{AAAA}+\frac{|\omega_0|+\varepsilon\|\overline{Y}_1\|}{|\Gamma^0_1({a}, {Z})\Gamma^0_1({a'}, {Z'})|}\left|\Gamma^0_1({a}, {Z})-\Gamma^0_1({a'}, {Z'})\right|.
\end{split}
\end{equation}

By \eqref{es0} and \eqref{R1}, with $\Gamma^0_1({a}, {Z})$, $\Gamma^0_1({a'}, {Z'})$ $\in I^0$, we have,
\begin{equation}\label{R2}
\begin{split}
&\left|\Gamma^0_2({a}, {Z})-\Gamma^0_2({a'}, {Z'})\right|\\
&\phantom{AA}\leq\frac{\varepsilon|\omega_0|+\varepsilon^2\|\overline{Y}_1\|+\varepsilon|\Gamma^0_1({a'}, {Z'})|}{|\Gamma^0_1({a}, {Z})\Gamma^0_1({a'}, {Z'})|}\bigg(\|D\overline{Y}_1\|B^0 \|\overline{r\circ K}\||{a}-{a'}|\\
&\phantom{AAAAAA}+\|D\overline{Y}_1\|\left(2+B^0 |{a'}|\|D(\overline{r\circ K})\|\right)\|{Z}-{Z'}\|\bigg)
\end{split}
\end{equation}

For the third term, similar to what we have done before, we add and subtract terms, then use the mean value theorem to get the estimate
\begin{equation}\label{es00}
\begin{split}
&\left|\overline{Y}_2({Z}(\theta-{a} t),\widetilde{{Z}}(\theta-{a} t;{a}), \varepsilon) - \overline{Y}_2({Z'}(\theta-{a'} t),\widetilde{{Z'}}(\theta-{a'} t;{a'}), \varepsilon)\right|\\
&\phantom{AA}\leq2\|D\overline{Y}_2\|\|Z-Z'\|+2t\|D\overline{Y}_2\|\|D{Z'}\| |{a}-{a'}|+\|D\overline{Y}_2\| \|D{Z'}\| \|r\circ K\||{a}-{a'}|\\
&\phantom{AAAA}+\|D\overline{Y}_2\|\|D{Z'}\| |{a'}|\|D(\overline{r\circ K})\|\|Z-Z'\|\\
&\phantom{AAAA}+t\|D\overline{Y}_2\|\|D{Z'}\|^2 |{a'}|\|D(\overline{r\circ K})\| |{a}-{a'}|\\
&\phantom{AA}\leq \|D\overline{Y}_2\|\big(2+B^0 |{a'}|\|D(\overline{r\circ K})\|\big)\|Z-Z'\|\\
&\phantom{AAAA}+B^0 \|D\overline{Y}_2\|\|\overline{r\circ K}\||{a}-{a'}|+tB^0\|D\overline{Y}_2\|\big(2+B^0 |{a'}|\|D(\overline{r\circ K})\|\big) |{a}-{a'}|.
\end{split}
\end{equation}
Where $\|D\overline{Y}_2\|$ is defined similarly to \eqref{dynorm}. Then,
\begin{equation}\label{R3}
\begin{split}
&\left|\Gamma^0_3({a}, {Z}),-\Gamma^0_3({a'}, {Z'})\right|\\
&\phantom{AA}\leq\varepsilon\|D\overline{Y}_2\|B^0\big(\frac{1}{\lambda_0^2}(2+B^0|{a'}|\|D(\overline{r\circ K})\|)-\frac{\|\overline{r\circ K}\|}{\lambda_0}\big) |{a}-{a'}|\\
&\phantom{AAAA}-\frac{\varepsilon}{\lambda_0}\|D\overline{Y}_2\|\big(2+B^0 |{a'}|\|D(\overline{r\circ K})\|\big)\|Z-Z'\|.
\end{split}
\end{equation}

With above estimates for each terms \eqref{R1}, \eqref{R2}, and \eqref{R3}, we have that for  the distance defined in \eqref{dist}, $d\big(\Gamma^0({a}, {Z}), \Gamma^0({a'}, {Z'})\big)$ is smaller than the sums of the right hand sides of \eqref{R1}, \eqref{R2}, and \eqref{R3}. More precisely,

\[
d\big(\Gamma^0(\omega, {Z}), \Gamma^0(\omega_2, {Z'})\big)\leq c_1 |{a}-{a'}|+c_2 \|Z-Z'\|
\]
Where
\begin{equation*}
\begin{split}
c_1=\varepsilon B^0\|\overline{r\circ K}&\|\left(\|D\overline{Y}_1\|\big(1+\frac{|\omega_0|+\varepsilon\|\overline{Y}_1\|+|\Gamma^0_1({a'}, {Z'})|}{|\Gamma^0_1({a}, {Z})\Gamma^0_1({a'}, {Z'})|}\big)-\frac{\|D\overline{Y}_2\|}{\lambda_0}\right)\\
&+\varepsilon \frac{B^0}{\lambda_0^2} \|D\overline{Y}_2\|\big(2+B^0|{a'}|\|D(\overline{r\circ K})\|\big)
\end{split}
\end{equation*}
and
\[
c_2=\ep\big(2+B^0 |{a'}|\|D(\overline{r\circ K})\|\big)\left(\|D\overline{Y}_1\|\big(1+\frac{|\omega_0|+\varepsilon\|\overline{Y}_1\|+|\Gamma^0_1({a'}, {Z'})|}{|\Gamma^0_1({a}, {Z})\Gamma^0_1({a'}, {Z'})|}\big)-\frac{\|D\overline{Y}_2\|}{\lambda_0}\right).
\]

Since $a$, ${a'}$, $\Gamma^0_1({a}, {Z})$, and $\Gamma^0_1({a'}, {Z'})$ are all in $I^0$, we have 
\begin{equation*}
\begin{split}
c_1\leq\varepsilon B^0\|\overline{r\circ K}&\|\left(\|D\overline{Y}_1\|\big(1+\frac{4|\omega_0|+4\varepsilon\|\overline{Y}_1\|+6|\omega_0|}{|\omega_0|^2}\big)-\frac{\|D\overline{Y}_2\|}{\lambda_0}\right)\\
&+\varepsilon \frac{B^0}{\lambda_0^2} \|D\overline{Y}_2\|\big(2+B^0|{a'}|\|D(\overline{r\circ K})\|\big),
\end{split}
\end{equation*}
and 
\[
c_2\leq \ep\big(2+B^0 |{a'}|\|D(\overline{r\circ K})\|\big)\left(\|D\overline{Y}_1\|\big(1+\frac{4|\omega_0|+4\varepsilon\|\overline{Y}_1\|+6|\omega_0|}{|\omega_0|^2}\big)-\frac{\|D\overline{Y}_2\|}{\lambda_0}\right)
\]
Because $c_1$ and $c_2$ are bounded by $\varepsilon$ multiplied by some constants, they can be made small with $\ep$ small. Therefore, if $\varepsilon$ is sufficiently small, we can find a $\mu_0<1$, such that \eqref{ctr0} is true, we have $\Gamma^0$ a contraction .
\end{proof}

Taking any initial guess $(\omega^0, W^{0,0}(\theta))\in D^0$. For example, one can take $\omega=\omega_0$, $W^{0,0}(\theta)= \left(\begin{smallmatrix}\theta\\0\end{smallmatrix}\right)$. Iterations of this initial guess under $\Gamma^0$ will have a limit by Lemma \ref{lem:contraction0}. Then by Lemma \ref{lem:prob0}, we can apply Lemma \ref{lem:Lan}, then we know that  the limit is in $D^0$. Therefore, we have a fixed point of $\Gamma^0$ in $D^0$, that is, there exist $\omega> 0$ and $W^0$ in $\C^{L+Lip}_0$ such that \eqref{inv0} is solved. Moreover, by the contraction argument, we know that the solution is unique. Therefore, $\omega$ is unique. $W^0$ is unique in the $\C^{L+Lip}_0$ space under the fixed phase $W^0_1(0)=0$. 

To prove the a-posteriori estimation part of Theorem \ref{thm:zero}, using $\Gamma^0$ is a contraction on $D^0$, we know that 
\begin{align}\label{d0}
d\big((\omega^0, W^{0,0}), (\omega, W^0)\big)&=\lim_{k\to\infty}d\big((\omega^0, W^{0,0}), (\Gamma^0)^k(\omega^0, W^{0,0})\big)\nonumber\\
&\leq \sum_{k=0}^{\infty}(\mu_0)^k d\big((\omega^0, W^{0,0}), \Gamma^0(\omega^0, W^{0,0})\big)\nonumber\\
&\leq\frac{1}{1-\mu_0}d\big((\omega^0, W^{0,0}), \Gamma^0(\omega^0, W^{0,0})\big).
\end{align}
It remains to estimate $d\big((\omega^0, W^{0,0}), \Gamma^0(\omega^0, W^{0,0})\big)$ by $\|E^0\|$, where the norm is the maximum norm defined in \eqref{norm}. We have
\begin{equation*}
E^0(\theta)=\omega^0 \frac{d}{d{\theta}}W^{0,0}(\theta)-
\begin{pmatrix}
\omega_0\\ \lambda_0W^{0,0}_2(\theta)
\end{pmatrix}-\varepsilon Y(W^{0,0}(\theta),\widetilde{W}^{0,0}(\theta;\omega^0), \varepsilon),
\end{equation*}
that is,
\begin{equation*}
\begin{pmatrix}
E^0_1(\theta)\\E^0_2(\theta)
\end{pmatrix}=\begin{pmatrix}
\omega^0 \frac{d}{d{\theta}}W^{0,0}_1(\theta)-\omega_0-\varepsilon \overline{Y}_1(W^{0,0}(\theta),\widetilde{W}^{0,0}(\theta;\omega^0), \varepsilon)\\
\omega^0 \frac{d}{d{\theta}}W^{0,0}_2(\theta)-\lambda_0 W^{0,0}_2(\theta)-\varepsilon \overline{Y}_2(W^{0,0}(\theta),\widetilde{W}^{0,0}(\theta;\omega^0), \varepsilon)
\end{pmatrix},
\end{equation*}
and,
\begin{equation*}
\begin{split}
&d\big((\omega^0, W^{0,0}), \Gamma^0(\omega^0, W^{0,0})\big)\\
&\phantom{AA}\leq\left|\omega_0+\varepsilon\int_0^1\overline{Y}_1(W^{0,0}(\theta), \widetilde{W}^{0,0}(\theta;\omega^0), \varepsilon) d\theta- \omega^0\right|\\
&\phantom{AAAA}+\sup_{\theta}\left|\frac{1}{\Gamma^0_1(\omega^0,~W^0)}\big(\omega_0\theta+\varepsilon\int_0^{\theta}\overline{Y}_1(W^{0,0}(\sigma), \widetilde{W}^{0,0}(\sigma;\omega^0), \varepsilon) d\sigma\big)-W^{0,0}_1(\theta)\right|\\
&\phantom{AAAA}+\sup_{\theta}\left|\varepsilon\int_0^{\infty}e^{\lambda_0 t} \overline{Y}_2(W^{0,0}(\theta-\omega^0 t),\widetilde{W}^{0,0}(\theta-\omega^0 t;\omega^0), \varepsilon) dt-W^{0,0}_2(\theta)\right|\\
&\phantom{AA}\leq\left|\int_0^1 E^0_1(\theta)d\theta\right|+\left|\int_0^{\infty}e^{\lambda_0 t}E^0_2(\theta-\omega^0 t)dt\right|\\
&\phantom{AAAA}+\frac{1}{|\Gamma^0_1(\omega^0,~W^0)|}\bigg(\left|\int^{\theta}_0 E^0_1(\sigma)d\sigma\right|+\|W^{0,0}_1\|\left|\int_0^1 E^0_1(\theta)d\theta\right|\bigg)\\
&\phantom{AA}\leq (1+\frac{2B^0}{|\omega_0|})\left|\int_0^1 E^0_1(\theta)d\theta\right|+\frac{2}{|\omega_0|}\left|\int^{\theta}_0 E^0_1(\sigma)d\sigma\right|+\left|\int_0^{\infty}e^{\lambda_0 t}E^0_2(\theta-\omega^0 t)dt\right|
\end{split}
\end{equation*}

For $\theta\in[0,1]$, we have 
\[
d\big((\omega^0, W^{0,0}), \Gamma^0(\omega^0, W^{0,0})\big)\leq \left(1+\frac{2+2B^0}{|\omega_0|}\right)\|E^0_1\|-\frac{1}{\lambda_0}\|E^0_2\|. 
\]
Combine this with the inequality \eqref{d0}, we have 
\begin{equation}\label{err0}
d\big((\omega^0, W^{0,0}), (\omega, W^0)\big)\leq \frac{1}{1-\mu_0}\left[\left(1+\frac{2+2B^0}{|\omega_0|}\right)\|E^0_1\|_{C_0}-\frac{1}{\lambda_0}\|E^0_2\|_{C_0}\right].
\end{equation}
By definition of the norm, \eqref{aprom} and $l=0$ case of \eqref{aprw0} are true for a constant $C$, which depends on $\varepsilon$, $B^0$, $\omega_0$, $\lambda_0$.

For other values of $l$,  one can use interpolation inequality in Lemma \ref{lem:inter}, to get
\begin{equation}\label{interp}
\begin{split}
\|W^{0,0}_1-W^{0}_1\|_{C^l}&\leq c(l,L)\|W^{0,0}_1-W^{0}_1\|_{C^0}^{1-\frac{l}{L}}\|W^{0,0}_1-W^{0}_1\|_{C^L}^{\frac{l}{L}}\\
&\leq c(l,L)\|W^{0,0}_1-W^{0}_1\|_{C^0}^{1-\frac{l}{L}}(2B^0)^{\frac{l}{L}}.
\end{split}
\end{equation}
Similar estimates can be done for the second component, this finishes the proof of the estimations in theorem \ref{thm:zero}.

For solution of the equation \eqref{inv0}, note that $K\circ W^0(\theta+\omega t)$ solves the equation \eqref{SDDE}:
\[
\frac{d}{dt}K\circ W^0(\theta+\omega t)=X(K\circ W^0(\theta+\omega t), K\circ W^0(\theta+\omega (t-r(K\circ W^0(\theta+\omega t))))).
\]
If $W^0$ is L times differentiable, then right hand side of above equation is L times differentiable, so is the left hand side. Using the fact that $K$ is an analytic local diffeomorphism, one can conclude that $W^0$ is (L+1) times differentiable. A bootstrap argument can be used to see $W^0$ is differentiable up to any order.

\subsection{Proof of Theorem \ref{thm:all}}\label{sc:piso}

With Theorem \ref{thm:zero}, $\omega$ and $W^0$ are known to us. To prove Theorem \ref{thm:all}, we have to consider the equations for the first order term, j-th order term, and then higher order term in $s$.  We will obtain $\lambda$, $W^1$ solving the first order equation \eqref{inv1}, $W^j$ solving \eqref{invj}, and then find $W^>$ which solves equation \eqref{invh}.

\subsubsection{First-order Equation}
\label{ssc:pw1}

Recall that for the first order term, we got an invariance equation \eqref{inv1}, see also below:

\[
\omega\frac{d}{d\theta}W^1(\theta)+\lambda W^1(\theta)-\begin{pmatrix}
0\\ \lambda_0W^1_2(\theta)
\end{pmatrix}=\varepsilon \overline{Y}^1(\theta,\lambda, W^0, W^1,\varepsilon),
\]
where
\begin{equation}\label{formy1}
\overline{Y}^1(\theta,\lambda,W^0,W^1,\varepsilon)={A}(\theta)W^1(\theta)+{B}(\theta;\lambda)W^1(\theta-\omega\overline{r\circ K}(W^0(\theta))),
\end{equation}
\begin{align}\label{formA}
&{A}(\theta)=-\omega D_2\overline{Y}(W^0(\theta),\widetilde{W}^0(\theta),\varepsilon)DW^0(\theta-\omega\overline{r\circ K}(W^0(\theta)))D(\overline{r\circ K})(W^0(\theta))\nonumber\\
&\phantom{AAAAAA}+D_1\overline{Y}(W^0(\theta),\widetilde{W}^0(\theta),\varepsilon)
\end{align}
and
\[
{B}(\theta;\lambda)=e^{-\lambda\overline{r\circ K}(W^0(\theta))}D_2\overline{Y}(W^0(\theta),\widetilde{W}^0(\theta),\varepsilon).
\]

Note that in the expression of ${A}$ and ${B}$ above, we suppressed the $\omega$ in the expression of $\widetilde{W}^0$. We do this to simplify the notation, since $\omega$ is already known from Theorem \ref{thm:zero}. 

\begin{remark}
Since $\overline{Y}^1(\theta,\lambda,W^0,W^1,\varepsilon)$, as in \eqref{formy1}, is linear in $W^1$, equation \eqref{inv1} for $W^1$, is linear and homogenous in $W^1$. Hence if $W^1(\theta)$ solves \eqref{inv1}, so does any scalar multiple of $W^1(\theta)$. 
\end{remark}

Componentwise, we have the following two equations:
\begin{align}
&\omega\frac{d}{d\theta}W^1_1(\theta)+\lambda W^1_1(\theta)=\varepsilon \overline{Y}^1_1(\theta,\lambda,W^0,W^1,\varepsilon),\label{inv11}
\\
&\omega\frac{d}{d\theta}W^1_2(\theta)+(\lambda-\lambda_0)W^1_2(\theta)=\varepsilon \overline{Y}^1_2(\theta,\lambda,W^0,W^1,\varepsilon).\label{inv12}
\end{align}

As already pointed out, for the unperturbed case, $W$ could be chosen as the identity map. So after we add a small perturbation, $W^1(\theta)\approx \left(\begin{smallmatrix}0\\1\end{smallmatrix}\right)$. We will be able to find a unique $W^1$ close to $\left(\begin{smallmatrix}0\\1\end{smallmatrix}\right)$ solving above equation \eqref{inv1}, by considering the following normalization:
\begin{equation}\label{normal}
\int^1_0W^1_2(\theta)d\theta=1.
\end{equation}

\begin{remark}
It is natural to choose above normalization \eqref{normal}, since under small perturbation, we have $W^1(\theta)\approx \left(\begin{smallmatrix}0\\1\end{smallmatrix}\right)$. Meanwhile, we believe that $\lambda$ does not depend on the choice of normalization as long as $\int^1_0W^1_2(\theta)d\theta\neq 0$.
\end{remark}

From now on, since $W^0$ is already known to us, we will omit $W^0$ from $\overline{Y}^1(\theta,\lambda,W^0,W^1,\varepsilon)$, and denote it as $\overline{Y}^1(\theta,\lambda,W^1,\varepsilon)$. We define an operator $\Gamma^1$ as follows:
\begin{equation}\label{op1}
\begin{split}
\Gamma^1\begin{pmatrix}
{b}\\
{F}_1 \\
{F}_2
\end{pmatrix}(\theta)&=
\begin{pmatrix}
\Gamma^1_1({b},{F})\\
\Gamma^1_2({b},{F})(\theta)\\
\Gamma^1_3({b},{F})(\theta)
\end{pmatrix}\\&=\begin{pmatrix}
\lambda_0+\varepsilon\int^1_0\overline{Y}^1_2(\theta,{b},{F},\varepsilon)d\theta\\
-\varepsilon\int^{\infty}_0 e^{bt}\overline{Y}^1_1(\theta+\omega t,{b},F,\varepsilon)dt\\
C(b, F)+\frac{\varepsilon}{\omega}\int^{\theta}_0 \overline{Y}^1_2(\sigma,{b},{F,}\varepsilon)-(\int^1_0\overline{Y}^1_2(\theta,{b},{F},\varepsilon)d\theta) F_2(\sigma)d\sigma
\end{pmatrix},
\end{split}
\end{equation}
where 
\begin{equation}
\begin{split}
C(b,F)&=1-\frac{\varepsilon}{\omega}\int^1_0\int^{\theta}_0 \overline{Y}^1_2(\sigma,b,F,\varepsilon)d\sigma d\theta\\
&\phantom{AA}+\frac{\varepsilon}{\omega}(\int^1_0\overline{Y}^1_2(\theta,b,F,\varepsilon)d\theta) \int^1_0\int^{\theta}_0 F_2(\sigma)d\sigma d\theta
\end{split}
\end{equation}
is a constant chosen to ensure that $\Gamma^1_3({b},{F})$ also satisfies the normalization condition \eqref{normal}, i.e. $\int^1_0 \Gamma^1_3({b},{F})(\theta)d\theta=1$.

Similar to previous section, section \ref{sc:plc}, for the domain of $\Gamma^1$, we consider a closed interval $I^1=\{{b}: |{b}-\lambda_0|\leq\frac{|\lambda_0|}{3}\}$, as well as the function space 
 \begin{equation}
 \begin{split}
\C^{L-1+Lip}_1=\{f\ |\ f:&{\T}\to{\T}\times\R,~f\text{ can be lifted to a function from } \R \text{ to } \R^2, \nonumber\\ 
&\text{still denoted as }f, \text{which satisfies }f(\theta+1)=f(\theta),\\&\|f\|_{L-1+Lip}\leq B^1, \text{and} \int^1_0 f_2(\theta)d\theta=1\},
\end{split}
\end{equation}
where 
\[
\|f\|_{L-1+Lip}=\max_{i=1,2, k=0,\dotsc,L-1}\{\sup_{\theta\in[0,1]}\|f^{(k)}_i(\theta)\|,Lip(f^{(L-1)}_i)\}.
\]
Where $L$ is as in Theorem \ref{thm:zero}, and $B^1$ is a positive constant.

Define $D^1=I^1\times \C^{L-1+Lip}_1$, then $\Gamma^1$ is defined on $D^1$. We have the following:

\begin{lemma}\label{lem:prob1}
If $\varepsilon$ is small enough, $\Gamma^1(D^1)\subset D^1$.
\end{lemma}

\begin{proof}
Since $\overline{Y}^1_2(\theta,{b},{F},\varepsilon)$ is bounded, for small $\varepsilon$,  we have $\Gamma^1_1({b},{F})\in I^1$.

Now consider $\Gamma^1_2({b},{F})(\theta)$, we first have to show that
\[
\Gamma^1_2({b},{F})(\theta+1)=\Gamma^1_2({b},{F})(\theta).
\]
This follows from the fact that $\overline{Y}^1_1(\theta+1,{b},F,\varepsilon)=\overline{Y}^1_1(\theta,{b},F,\varepsilon)$, which is true by periodicity of $W^0$ as in equation \eqref{pc0}, of $F$, and of $\overline{r\circ K}$ with respect to its first component. 

Now we check $\frac{d^n}{d\theta^n}\Gamma^1_2({b},{F})(\theta)$, $0\leq n\leq L-1$, is bounded. Notice that
\[
\frac{d^n}{d\theta^n}\Gamma^1_2({b},{F})(\theta)=-\ep\int^{\infty}_0e^{bt}\frac{\partial^n}{\partial\theta^n}\overline{Y}^1_1(\theta+\omega t,{b},F,\varepsilon)dt.
\]
 By dominated convergence theorem, it suffices to check that $\frac{\partial^n}{\partial\theta^n}\overline{Y}^1_1(\theta+\omega t,{b},F,\varepsilon)$ is bounded. If one uses Faà di Bruno's formula, as in Lemma \ref{faa},  boundedness of $\frac{\partial^n}{\partial\theta^n}\overline{Y}^1_1(\theta+\omega t,{b},F,\varepsilon)$ is ensured by assumptions on  $\overline{Y}$, $\overline{r\circ K}$, and $W^0(\theta)$, as well as $F\in \C^{L-1+Lip}_1$. Then for $\ep$ small enough, the derivatives could be bounded by $B^1$. Bound for Lipschitz constant of $\frac{d^{L-1}}{d\theta^{L-1}}\Gamma^1_2({b},{F})(\theta)$ also follows.
 
For $\Gamma^1_3({b},{F})(\theta)$, we will first show that it is periodic. Notice that
\begin{equation}\label{derivT3}
\frac{d}{d\theta}\Gamma^1_3({b},{F})(\theta)=\frac{\ep}{\omega}\overline{Y}^1_2(\theta,{b},{F},\varepsilon)-\frac{\ep}{\omega}\left(\int^1_0\overline{Y}^1_2(\theta,{b},{F},\varepsilon)d\theta\right) F_2(\theta)
\end{equation}
is periodic. Hence, to show periodicity of $\Gamma^1_3({b},{F})(\theta)$, it suffices to see that $\Gamma^1_3({b},{F})(0)=\Gamma^1_3({b},{F})(1)$, which is true because $\int^1_0F_2(\theta)d\theta=1$. The choice of the constant $C(b,F)$ ensures that the normalization condition $\int^1_0\Gamma^1_3({b},{F})(\theta)d\theta=1$ is also verified.

Take derivatives of \eqref{derivT3}, we have for $2\leq n\leq L-1$
\[
\frac{d^n}{d\theta^n}\Gamma^1_3({b},{F})(\theta)=\frac{\ep}{\omega}\bigg(\frac{d^{(n-1)}}{d\theta^{(n-1)}}\overline{Y}^1_2(\theta,{b},{F},\varepsilon)-\left(\int^1_0\overline{Y}^1_2(\theta,{b},{F},\varepsilon)d\theta\right)\frac{d^{(n-1)}}{d\theta^{(n-1)}}F_2(\theta)\bigg),
\]
which will be $\frac{\ep}{\omega}$ multiplied by bounded functions due to the assumptions on  $\overline{Y}$, $\overline{r\circ K}$, and $W^0(\theta)$, as well as $F\in \C^{L-1+Lip}_1$. When $\ep$ is small, they could all be bounded by $B^1$. Similar for Lipschitz constant of $\frac{d^{L-1}}{d\theta^{L-1}}\Gamma^1_3({b},{F})(\theta)$. 

Hence for $\varepsilon$ small enough, where the smallness condition depends on bounds of the derivatives of $\overline{Y}$, $\overline{r\circ K}$, $B^0$, and $B^1$, but not on the specific choice of $(b,F)\in D^1$, we have that $(\Gamma^1_2({b},{F}), \Gamma^1_3({b},{F}))\in \C^{L-1+Lip}_1$. This finishes the proof.
\end{proof}

Recall the distance introduced in \eqref{dist}:
\[
d(({a}, {Z}), ({a'}, {Z'}))=|{a}-{a'}|+\|Z-Z'\|,
\]
where
\[
\|Z-Z'\|=\max\left\{\sup_{\theta}|{Z}_1(\theta)-{Z'_1}(\theta)|, \sup_{\theta}|{Z}_2(\theta)-{Z'_2}(\theta)|\right\}.
\]
\begin{lemma}\label{lem:contraction1}
Under above defined distance on $D^1$, for small enough $\varepsilon$, $\Gamma^1$ is a contraction.
\end{lemma}
\begin{proof}
We will show that for $\varepsilon$ small enough, we can find a constant $0<\mu_1<1$ such that
\begin{equation}\label{ctr1}
d(\Gamma^1({b}, {F}), \Gamma^1({b'}, {F'}))<\mu_1 d(({b}, {F}), ({b'}, {F'})).
\end{equation}

Note that 
\begin{align}\label{dis1}
\begin{split}
d(&\Gamma^1({b}, {F}), \Gamma^1({b'}, {F'}))\\
&\leq\varepsilon\left|\int^1_0\overline{Y}^1_2(\theta,{b},{F},\varepsilon)-\overline{Y}^1_2(\theta,{b'},{F'},\varepsilon)d\theta\right|\\
&\phantom{A}+\varepsilon\sup_{\theta}\left|\int^{\infty}_0 e^{bt}\overline{Y}^1_1(\theta+\omega t,{b},F, \varepsilon)- e^{b't}\overline{Y}^1_1(\theta+\omega t,{b'},F',\varepsilon)dt\right|\\
&\phantom{A}+\frac{\varepsilon}{|\omega|}\sup_{\theta}\left|\int^{\theta}_0 \overline{Y}^1_2(\sigma,{b},{F}, \varepsilon)-\big(\int^1_0\overline{Y}^1_2(\theta,{b},{F},\varepsilon)d\theta\big) F_2(\sigma)d\sigma\right.\\
&\left.\phantom{AAAAAAAAA}-\int^{\theta}_0 \overline{Y}^1_2(\sigma,{b'},{F'}, \varepsilon)+\big(\int^1_0\overline{Y}^1_2(\theta,{b'},F',\varepsilon)d\theta\big) F'_2(\sigma)d\sigma\right|\\
&\phantom{A}+|C(F,b)-C(F',b')|
\end{split}
\end{align}

As before, we will consider each term of the right hand side of the above inequality \eqref{dis1}.

Recall that $\overline{Y}^1$ has the form, \eqref{formy1}
\[
\overline{Y}^1(\theta,\lambda,W^1,\varepsilon)={A}(\theta)W^1(\theta)+{B}(\theta;\lambda)W^1(\theta-\omega\overline{r\circ K}(W^0(\theta))).
\]
If we use notation:
\[
A(\theta)= \begin{pmatrix}
 A_{11}(\theta)&A_{12}(\theta)\\
 A_{21}(\theta)&A_{22}(\theta)
 \end{pmatrix},~~~~~~~~B(\theta;\lambda)= \begin{pmatrix}
 B_{11}(\theta;\lambda)&B_{12}(\theta;\lambda)\\
 B_{21}(\theta;\lambda)&B_{22}(\theta;\lambda)\end{pmatrix},
\]
 then 
 \begin{equation*}
 \begin{split}
\overline{Y}^1_1(\theta,\lambda,W^1,\varepsilon)=&A_{11}(\theta)W^1_1(\theta)+A_{12}(\theta)W^1_2(\theta)\\&+B_{11}(\theta;\lambda)W^1_1(\theta-\omega \overline{r\circ K}(W^0(\theta)))\\&+B_{12}(\theta;\lambda)W^1_2(\theta-\omega \overline{r\circ K}(W^0(\theta))),
 \end{split}
 \end{equation*}
 and
 \begin{equation*}
 \begin{split}
\overline{Y}^1_2(\theta,\lambda,W^1,\varepsilon)=&A_{21}(\theta)W^1_1(\theta)+A_{22}(\theta)W^1_2(\theta)\\&+B_{21}(\theta;\lambda)W^1_1(\theta-\omega \overline{r\circ K}(W^0(\theta)))\\&+B_{22}(\theta;\lambda)W^1_2(\theta-\omega \overline{r\circ K}(W^0(\theta))).
 \end{split}
 \end{equation*}

We estimate 
\begin{align*}
&|B(\theta;b)|\leq e^{-\frac{4}{3}\lambda_0\|\overline{r\circ K}\|}\|D_2\overline{Y}\|,\\
\intertext{and}
&|B(\theta;b)-B(\theta;b')|\leq \|D_2\overline{Y}\| e^{-\frac{4}{3}\lambda_0\|\overline{r\circ K}\|}\|\overline{r\circ K}\||b-b'|.
\end{align*}

Also, if we define $\|A\| = \max_{\theta}\|A(\theta)\|$, where $\|A(\theta)\|$ is the operator norm corresponding to the maximum norm $\|\cdot\|$ defined as in equation \eqref{norm}. Then,
\begin{equation*} 
\begin{split}
&|\overline{Y}^1_1(\theta,{b},{F},\varepsilon)-\overline{Y}^1_1(\theta,{b'},{F'},\varepsilon)|\\
&\phantom{A}\leq \|A\|\|F-F'\|+\|B(\theta;b)\|\|F-F'\|+\|B(\theta;b)-B(\theta;b')\|\|F'\|\\
&\phantom{A}\leq (\|A\|+e^{-\frac{4}{3}\lambda_0\|\overline{r\circ K}\|}\|D_2\overline{Y}\|)\|F-F'\|+B^1 \|D_2\overline{Y}\| e^{-\frac{4}{3}\lambda_0\|\overline{r\circ K}\|}\|\overline{r\circ K}\||b-b'|,
\end{split}
\end{equation*}
and similarly, 
\begin{equation*} 
\begin{split}
&|\overline{Y}^1_2(\theta,{b},{F},\varepsilon)-\overline{Y}^1_2(\theta,{b'},{F'},\varepsilon)|\\
&\phantom{A}\leq (\|A\|+e^{-\frac{4}{3}\lambda_0\|\overline{r\circ K}\|}\|D_2\overline{Y}\|)\|F-F'\|+B^1 \|D_2\overline{Y}\| e^{-\frac{4}{3}\lambda_0\|\overline{r\circ K}\|}\|\overline{r\circ K}\||b-b'|.
\end{split}
\end{equation*}

Note also that 
\begin{equation*}
|\overline{Y}^1_1(\theta,{b},{F},\varepsilon)|\leq B^1(\|A\|+e^{-\frac{4}{3}\lambda_0\|\overline{r\circ K}\|}\|D_2\overline{Y}\|),
\end{equation*}
similarly,
\begin{equation*}
|\overline{Y}^1_2(\theta,{b},{F},\varepsilon)|\leq B^1(\|A\|+e^{-\frac{4}{3}\lambda_0\|\overline{r\circ K}\|}\|D_2\overline{Y}\|).
\end{equation*}

Now for the first term in \eqref{dis1}, we have 
\begin{align*}
\left|\Gamma^1_1({b}, {F})-\Gamma^1_1({b'}, {F'})\right|&\leq
\varepsilon(\|A\|+e^{-\frac{4}{3}\lambda_0\|\overline{r\circ K}\|}\|D_2\overline{Y}\|)\|F-F'\|\\
&\phantom{AA}+\varepsilon B^1 \|D_2\overline{Y}\| e^{-\frac{4}{3}\lambda_0\|\overline{r\circ K}\|}\|\overline{r\circ K}\||b-b'|.
\end{align*}

For the second term in \eqref{dis1}, we have for all $\theta$,
\begin{align*}
&\left|\Gamma^1_2({b}, {F})-\Gamma^1_2({b'}, {F'})\right|\leq\\
&\phantom{AA}-\frac{3\varepsilon}{2\lambda_0}(\|A\|+e^{-\frac{4}{3}\lambda_0\|\overline{r\circ K}\|}\|D_2\overline{Y}\|)\|F-F'\|\\
&\phantom{AA}-\frac{3B^1\ep}{2\lambda_0}\left(e^{-\frac{4}{3}\lambda_0\|\overline{r\circ K}\|}\|D_2\overline{Y}\|\big(\|\overline{r\circ K}\|-\frac{3}{2\lambda_0}\big)-\frac{3}{2\lambda_0}\|A\|\right)|b-b'|
\end{align*}

For the third term in \eqref{dis1}, we have
\begin{align*}
\left|\Gamma^1_3({b}, {F})-\Gamma^1_3({b'}, {F'})\right|
\leq& \frac{\varepsilon}{|\omega|}(1+2B^1)(\|A\|+e^{-\frac{4}{3}\lambda_0\|\overline{r\circ K}\|}\|D_2\overline{Y}\|)\|F-F'\|\\
&+\frac{B^1\ep}{|\omega|}(1+B^1)\|D_2\overline{Y}\| e^{-\frac{4}{3}\lambda_0\|\overline{r\circ K}\|}\|\overline{r\circ K}\||b-b'|
\end{align*}

Similar holds for the last part in \eqref{dis1},
\begin{align*}
|C(F,b)-C(F',b')|
\leq& \frac{\varepsilon}{|\omega|}(1+2B^1)(\|A\|+e^{-\frac{4}{3}\lambda_0\|\overline{r\circ K}\|}\|D_2\overline{Y}\|)\|F-F'\|\\
&+\frac{B^1\ep}{|\omega|}(1+B^1)\|D_2\overline{Y}\| e^{-\frac{4}{3}\lambda_0\|\overline{r\circ K}\|}\|\overline{r\circ K}\||b-b'|
\end{align*}

Combine all the estimations above, we can find constants $c_1$, $c_2$ such that, \[
d(\Gamma^1({b}, {F}), \Gamma^1({b'}, {F'}))\leq \ep(c_1|b-b'|+c_2\|F-F'\|).
\]
Therefore, for small enough $\varepsilon$, we will have a contraction, so that we can find a $\mu_1$ such that equation \eqref{ctr1} is true.
\end{proof}

Taking any initial guess $(\lambda^0,W^{1,0})\in D^1$, we could take $\lambda^0=\lambda_0$ and $W^{1,0}(\theta)= \left(\begin{smallmatrix}0\\1\end{smallmatrix}\right)$, the sequence $(\Gamma^1)^n(\lambda^0, W^{1,0})$ has a limit in $D^1$, we denote it by $(\lambda, W^1)$. $(\lambda, W^1)$ is a fixed point of operator $\Gamma^1$, hence it solves equation \eqref{inv1}. Since the operator is a contraction, $\lambda$ is unique, $W^1$ is unique in $C^0$ sense under the normalization condition \eqref{normal}.

Similar to what we have done in estimation \eqref{d0} in section \ref{sc:plc}, notice that
\begin{equation}\label{d1}
d\big((\lambda^0,W^{1,0}),(\lambda, W^1)\big)\leq \frac{1}{1-\mu_1}d\big((\lambda^0,W^{1,0}),\Gamma^1(\lambda^0,W^{1,0})\big).
\end{equation}

We will estimate $d\big((\lambda^0,W^{1,0}),\Gamma^1(\lambda^0,W^{1,0})\big)$ by $\|E^1\|$. If we write $E^1(\theta)$ in matrix form, we have
\begin{equation*}
\begin{pmatrix}
E^1_1(\theta)\\
E^1_2(\theta)
\end{pmatrix}=\begin{pmatrix}
\omega\frac{d}{d\theta}W^{1,0}_1(\theta)+\lambda^0 W^{1,0}_1(\theta)-\varepsilon \overline{Y}^1_1(\theta,\lambda^0,W^{1,0},\varepsilon)\\
\omega\frac{d}{d\theta}W^{1,0}_2(\theta)+(\lambda^0-\lambda_0)W^{1,0}_2(\theta)-\varepsilon \overline{Y}^1_2(\theta,\lambda^0,W^{1,0},\varepsilon)
\end{pmatrix}.
\end{equation*}

Therefore,
\begin{align*}
d\big((\lambda^0,W^{1,0}), &\Gamma^1(\lambda^0,W^{1,0})\big)\\
\leq&|\lambda_0+\varepsilon\int^1_0\overline{Y}^1_2(\theta,{\lambda^0},{W^{1,0}},\varepsilon)d\theta-\lambda^0|\\
&+\sup_{\theta}\left|W^{1,0}_1(\theta)+\varepsilon\int^{\infty}_0 e^{\lambda^0t}\overline{Y}^1_1(\theta+\omega t,{\lambda^0},W^{1,0},\varepsilon)dt\right|\\
&+\sup_{\theta}\left|C(\lambda^0, W^{1,0})+\frac{\varepsilon}{\omega}\int^{\theta}_0 \overline{Y}^1_2(\sigma,{\lambda^0},{W^{1,0},}\varepsilon)\right.\\
&\phantom{AAAAA}\left.-\left(\int^1_0\overline{Y}^1_2(\theta,{\lambda^0},{W^{1,0}},\varepsilon)d\theta\right)W^{1,0}_2(\sigma)d\sigma-W^{1,0}_2(\theta)\right|\\
\leq& \left|\int_0^1E^1_2(\theta)d\theta\right|+\left|\int^{\infty}_0 e^{\lambda^0t}E^1_1(\theta+\omega t)dt\right|+\frac{2+2B^1}{|\omega|}\|E^1_2\|\\
\leq& \frac{1}{|\lambda^0|}\|E^1_1\|+\left(1+\frac{2+2B^1}{|\omega|}\right)\|E^1_2\|\\
\leq& \frac{3}{2|\lambda_0|}\|E^1_1\|+\left(1+\frac{2+2B^1}{|\omega|}\right)\|E^1_2\|.
\end{align*}

Then
\begin{equation}\label{err1}
d\big((\lambda^0,W^{1,0}),(\lambda, W^1)\big)\leq\frac{1}{1-\mu_1} \left[\frac{3}{2|\lambda_0|}\|E^1_1\|+\left(1+\frac{2+2B^1}{|\omega|}\right)\|E^1_2\|\right].
\end{equation}

Therefore, we can find a constant $C$, depending on $\varepsilon$, $B^1$, $\omega$ and $\lambda_0$ such that $|\lambda-\lambda^0|\leq C\|E^1\|$. This proves \eqref{aprlam}.

\subsubsection{Equation for jth order terms}\label{ssc:pwj}

For each $j\geq 2$, we can proceed in a similar manner to find $W^j$. With $\omega$, $\lambda$, $W^0$, and $W^1$ known, Equations for $W^j$'s are easier to analyze. 
\begin{remark}
As we will see, for theoretical result, we can stop at order 1 and start to deal with the higher order term. We include here the discussion for $W^j$'s for numerical interests.
\end{remark}
Assume now that we have already obtained $W^0,\dotsc,W^{j-1}$, and $\omega$, $\lambda$, we are going to find $W^j(\theta)$. To obtain the invariance equation satisfied by $W^j$, which was in equation \eqref{invj}. We consider the j-th order terms in the equation \eqref{invv}. 
Note that the coefficient for $s^j$ in $\widetilde{W}(\theta,s)$, is
\[
-\omega DW^0(\theta-\omega \overline{r\circ K}(W^0(\theta))D(\overline{r\circ K})(W^0(\theta))W^j(\theta)
\]
Therefore, $\overline{Y}^j$ is of the form:
\begin{equation}\label{formyj}
\overline{Y}^j(\theta,W^0,W^j,\ep)=A(\theta)W^j(\theta),
\end{equation}
where $A(\theta)$ is the same as in \eqref{formA},
\begin{align*}
A(\theta)=-\omega& D_2\overline{Y}(W^0(\theta),\widetilde{W}(\theta),\ep)DW^0(\theta-\omega \overline{r\circ K}(W^0(\theta))D(\overline{r\circ K})(W^0(\theta))\\
&+D_1\overline{Y}(W^0(\theta),\widetilde{W}(\theta),\ep).
\end{align*}

We also note that $R^j(\theta)$ will be some expression in the derivatives of $\overline{Y}$ evaluated at $(W^0(\theta),\widetilde{W}(\theta),\ep)$, multiplied with $W^0,\dotsc,W^{j-1}$. Therefore, $R^j(\theta)$ will have the same regularity as $W^{j-1}$. We will see inductively by the following argument that $W^j$ is $(L-1)$ times differentiable with $(L-1)$-th derivative Lipschitz.

From now on, we will write $\overline{Y}^j$ as $\overline{Y}^j(\theta,W^j,\ep)$, for that $\lambda$ and $W^0$ are known to us. Componentwisely, $W^j$ should satisfy
\begin{align}
&\omega\frac{d}{d\theta}W^j_1(\theta)+\lambda jW^j_1(\theta)=\ep\overline{Y}^j_1(\theta,W^j,\ep)+R^j_1(\theta),\label{invj1}\\
&\omega\frac{d}{d\theta}W^j_2(\theta)+(\lambda j-\lambda_0)W^j_2(\theta)=\ep\overline{Y}^j_2(\theta,W^j,\ep)+R^j_2(\theta)\label{invj2}.
\end{align}

For functions in the space
 \begin{equation}
 \begin{split}
\C^{L-1+Lip}_j=\{f\ |\ f:&{\T}\to{\T}\times\R,~f\text{ can be lifted to a function from } \R \text{ to } \R^2, \nonumber\\ 
&\text{still denoted as }f, \text{which satisfies }f(\theta+1)=f(\theta),\\&\|f\|_{L-1+Lip}\leq B^j\},
\end{split}
\end{equation}
where 
\[
\|f\|_{L-1+Lip}=\max_{i=1,2, k=0,\dotsc,L-1}\{\sup_{\theta\in[0,1]}\|f^{(k)}_i(\theta)\|,Lip(f^{(L-1)}_i)\}.
\]

Similar to what we have done above, define an operator on space $\C^{L-1+Lip}_j$
\begin{equation}\label{opj}
\Gamma^j(G)(\theta)=\begin{pmatrix}
-\ep\int^{\infty}_0e^{\lambda jt}\left(\overline{Y}^j_1(\theta+\omega t, G,\ep)+R^j_1(\theta+\omega t)\right)dt\\
-\ep\int^{\infty}_0e^{(\lambda j-\lambda_0)t}\left(\overline{Y}^j_2(\theta+\omega t, G,\ep)+R^j_2(\theta+\omega t)\right)dt
\end{pmatrix}
\end{equation}

Assume that we have already obtained $W^k$ in $\C^{L-1+Lip}_k$ for $k=0,\dotsc,j-1$, we have the following:

\begin{lemma}
For small enough $\ep$, we have $\Gamma^j(\C^{L-1+Lip}_j)\subset \C^{L-1+Lip}_j$.
\end{lemma}
This follows from $\lambda< 0$ and $(\lambda j-\lambda_0)< 0$ for $j\geq 2$ and the regularity of $W^0,\dotsc, W^j$, $\overline{Y}^j$, and $R^j$. Moreover, we have $\ep$ in front of the expression. Since this is very similar to the analysis of $W^0$ and $W^1$, we will omit the detailed proof here.

We also know that $\Gamma^j$ is a $C^0$ contraction for small $\ep$.
\begin{lemma}
For small enough $\ep$, $\Gamma^j$ is a contraction in $C^0$ distance.
\end{lemma}
This follows easily from that $\lambda< 0$ and $(\lambda j-\lambda_0)< 0$ for $j\geq 2$, and $\overline{Y}^j$ is linear in $W^j$.

If we define norm as before
\[
\|G\|=\max\{\sup_{\theta}|G_1(\theta)|, \sup_{\theta}|G_2(\theta)|\},
\]
above lemma tells us that, if $\ep$ is small enough, then one can find $0<\mu_j<1$, such that
\[
\|\Gamma(G)-\Gamma(G')\|\leq\mu_j\|G-G'\|.
\]

Taking any initial guess $W^{j,0}\in \C^{L-1+Lip}_j$, we would take $W^{j,0}(\theta)= \left(\begin{smallmatrix}0\\0\end{smallmatrix}\right)$, the sequence $(\Gamma^j)^n(W^{j,0})$ has a limit in $\C^{L-1+Lip}_j$, we denote it by $W^j$. $W^j$ is a fixed point of operator $\Gamma^j$, so it solves equation \eqref{invj}. $W^j$ close to the initial guess, is unique in the sense of $C^0$ by the contraction argument. We will see quantitative estimates below.

We know that
\begin{equation}
\|W^j-W^{j,0}\|\leq \frac{1}{1-\mu_j}\|W^{j,0}-\Gamma^j(W^{j,0})\|.
\end{equation}

With similar argument as in the error estimation of $W^0$ and $W^1$, we have
\begin{align*}
|W^{j,0}_1(\theta)-\Gamma^j_1(W^{j,0})(\theta)|&\leq-\frac{1}{j\lambda}\|E^j_1\|,\\
|W^{j,0}_2(\theta)-\Gamma^j_2(W^{j,0})(\theta)|&\leq-\frac{1}{j\lambda-\lambda_0}\|E^j_2\|.
\end{align*}
Therefore, we have
\begin{equation}\label{errj}
\|W^j-W^{j,0}\|\leq \frac{1}{1-\mu_j}\left(-\frac{1}{j\lambda}\|E^j_1\|-\frac{1}{j\lambda-\lambda_0}\|E^j_2\|\right)\leq C\|E^j\|.
\end{equation}
We stress that above $C$ depends on $j$, $\ep$, $B^j$ and the SDDE, however, it does not depend on choice of $W^{j,0}$ in space $\C^{L-1+Lip}_j$.
\subsubsection{Equation of Higher Order Term}
\label{ssc:phot}
Now we have already found $\omega$, $\lambda$, $W^0,\dotsc, W^{N-1}$. It remains to consider the higher order term. We will solve equation \eqref{invh} locally in this section, which will establish the existence in Theorem \ref{thm:all}. From now on, we will write:
\begin{equation}\label{whform}
W(\theta, s)=W^{\leq}(\theta, s)+W^>(\theta, s),
\end{equation}
where $W^{\leq}(\theta, s)=\sum^{N-1}_{j=0}W^j(\theta)s^j$. 
To make the analysis feasible, we do a cut-off to the equation satisfied by $W^>$ in \eqref{invh}:
\begin{equation}\label{invhh}
(\omega\partial_{\theta}+s\lambda\partial_s)W^>(\theta,s)=\begin{pmatrix}
0\\\lambda_0W^>_2(\theta, s)
\end{pmatrix}+\varepsilon Y^>(W^>,\theta, s, \varepsilon)\phi(s),
\end{equation}
where 
\begin{equation}\label{formy>}
Y^>(W^>,\theta, s, \varepsilon)=\overline{Y}(W(\theta,s),\widetilde{W}(\theta,s),\varepsilon)-\sum_{i=0}^{N-1}\overline{Y}^i(\theta)s^i,
\end{equation}
\begin{equation*}
\overline{Y}^i(\theta)=\frac{1}{i!}\frac{\partial^i}{\partial s^i}(\overline{Y}(W(\theta,s),\widetilde{W}(\theta,s),\varepsilon))|_{s=0},
\end{equation*}
and recall the $C^{\infty}$ cut-off function $\phi:\mathbb{R}\to[0,1]$ as introduced in \eqref{cutoff}:
\begin{equation*}
\phi(x)=
\begin{cases}
1 & \text{if}\quad |x|\leq\frac{1}{2},\\
0 & \text{if}\quad |x|>1.
\end{cases}
\end{equation*}

\begin{remark}
Cut-off is needed in our method. We note that similar to before, the boundaries for cut-off function above($\frac{1}{2}$ and $1$) could be changed to any positive numbers $a_1<a_2$. 

Adding a cut-off is not too restrictive. Indeed, we only get local results for the original problem near the limit cycle. Since we have used extensions to get the prepared equation \eqref{invv}, what happens for $s$ with large absolute value will not matter. 
\end{remark}

Now let $c(t)=(\theta+\omega t, se^{\lambda t})$ be the characteristics, we define an operator as follows:
\begin{equation}\label{op>}
\Gamma^>(H)(\theta, s)=-\varepsilon\int^{\infty}_0 \begin{pmatrix}
1&0\\
0&e^{-\lambda_0t}
\end{pmatrix}Y^>(H,c(t), \varepsilon)\phi(se^{\lambda t})dt.
\end{equation}

If there is a fixed point of $\Gamma^>$ which has some regularity, it will solve the modified invariance equation \eqref{invhh}. For the domain of $\Gamma^>$, assume $L^>$ is a positive integer, we consider $D^>$ the space of functions $H:\mathbb{T}\times\mathbb{R}\to\mathbb{T}\times\mathbb{R}$, where $\partial^l_{\theta}\partial^m_sH_i(\theta, s)$, $i=1, 2$, exists if $l+m\leq L^>$, with $\|\cdot\|_{L^>,N}$ norm bounded by a constant ${B}$:

\begin{equation}\label{>norm}
\|H\|_{L^>,N}:=\max_{l+m\leq L^>, i=1,2}
\begin{cases}
\sup_{(\theta, 
s)\in\mathbb{T}\times\mathbb{R}}|\partial^l_{\theta}
\partial^m_sH_i(\theta, s)||s|^{-(N-m)} & \text{if} \quad m\leq N,\\
\sup_{(\theta, 
s)\in\mathbb{T}\times\mathbb{R}}|\partial^l_{\theta}
\partial^m_sH_i(\theta, s)| & \text{if}\quad m>N.
\end{cases}
\end{equation}
Under above notations in \eqref{whform}, we have 
\begin{equation*}
 \begin{split}
  \widetilde{W}(\theta,s)
  &=W(\theta-\omega \overline{r\circ K}(W(\theta,s)),se^{-\lambda \overline{r\circ K}(W(\theta,s))}) 
\\
  &=W^{\leq}(\theta-\omega \overline{r\circ K}(W(\theta,s)),se^{-\lambda \overline{r\circ K}(W(\theta,s))}) \\
  & \quad +W^{>}(\theta-\omega \overline{r\circ K}(W(\theta,s)),se^{-\lambda \overline{r\circ K}(W(\theta,s))}).
 \end{split}
\end{equation*}

We define 
\begin{equation}\label{tildew>}
\widetilde{W}^>(\theta,s)=W^{>}(\theta-\omega \overline{r\circ K}((W^{\leq}+W^>)(\theta,s)),se^{-\lambda \overline{r\circ K}((W^{\leq}+W^>)(\theta,s))}).
\end{equation}

\begin{lemma}\label{lem:prob>}
If $\varepsilon$ is small enough, $\Gamma^>(D^>)\subset D^>$.
\end{lemma}

\begin{proof}
For $H\in D^>$, we need to prove that for $i=1, 2$, and $l+m\leq L^>$, $\partial^l_{\theta}\partial^m_s\Gamma^>_i(H)(\theta, s)$ exists,  also that $\|\Gamma^>(H)\|_{L^>,N}$ is bounded by $B$. Using definition in equation \eqref{tildew>}
\begin{equation*}
\widetilde{H}(\theta, s)=H(\theta-\omega \overline{r\circ K}((W^{\leq}+H)(\theta,s)),se^{-\lambda \overline{r\circ K}((W^{\leq}+H)(\theta,s))})
\end{equation*}

We first claim that for $\|H\|_{L^>,N}\leq {B}$, we can find $C$, which does not depend on the choice of $H$, such that for $l+m\leq L^>$, $i=1,2$, $(\theta, s)\in \widetilde{\T}\times[-1,1]$:
\begin{equation}\label{tildeh}
\begin{cases}
|\partial^l_{\theta}\partial^m_s\widetilde{H}_i(\theta, s)|\leq C|s|^{(N-m)} 
& \text{if} \quad m\leq N,\\
|\partial^l_{\theta}\partial^m_s\widetilde{H}_i(\theta, s)|\leq C 
& \text{if}\quad m>N.
\end{cases}
\end{equation}

Note that within the proof of this lemma, $C$ may vary from line to line. Finally, we will take $C$ to be the maximum of all $C$s appear in this proof.

To prove above claim, notice that $\|H\|_{L^>,N}\leq {B}$ implies that 
\[
\begin{cases}
|\partial^l_{\theta}\partial^m_sH_i(\theta, s)|\leq {B}|s|^{(N-m)} & \text{if} 
\quad m\leq N,\\
|\partial^l_{\theta}\partial^m_sH_i(\theta, s)|\leq {B} & \text{if}\quad m>N.
\end{cases}
\]
 for $l+m\leq L^>$, $i=1,2$, and $(\theta,s)\in \mathbb{T}\times \mathbb{R}$. Then
\[
|\widetilde{H}_i(\theta,s)|\leq B|s|^{N}e^{-\lambda N \overline{r\circ K}((W^{\leq}+H)(\theta,s))}.
\]
By boundedness of $\overline{r\circ K}$, we have that $|\widetilde{H}_i(\theta,s)|\leq C|s|^{N}$. Note that
\begin{align*}
\frac{\partial}{\partial \theta}\widetilde{H}_i(\theta,s)=\partial_{\theta}H_i\left(\theta-\omega \overline{r\circ K}((W^{\leq}+H)(\theta,s)),se^{-\lambda \overline{r\circ K}((W^{\leq}+H)(\theta,s))}\right)\cdot\\
\cdot\left(1-\omega D(\overline{r\circ K})((W^{\leq}+H)(\theta,s))\partial_{\theta}(W^{\leq}+H)(\theta,s)\right)\phantom{AAA}\\
+\partial_{s}H_i\left(\theta-\omega \overline{r\circ K}((W^{\leq}+H)(\theta,s)),se^{-\lambda\overline{r\circ K}((W^{\leq}+H)(\theta,s))}\right)\cdot\\
\cdot s(-\lambda)D(\overline{r\circ K})((W^{\leq}+H)(\theta,s))\partial_{\theta}(W^{\leq}+H)(\theta,s)e^{-\lambda \overline{r\circ K}((W^{\leq}+H)(\theta,s))}
\end{align*}

Then, we have
\begin{align*}
\left|\frac{\partial}{\partial \theta}\widetilde{H}_i(\theta,s)\right|\leq&
B|s|^{N}e^{-\lambda N \|\overline{r\circ K}\|}(1+|\omega| \|D(\overline{r\circ K})\|\|\partial_{\theta}(W^{\leq}+H)\|\\
&+B|s|^{N-1}e^{-\lambda (N-1) \|\overline{r\circ K}\|}|s||\lambda|\|D(\overline{r\circ K})\|e^{-\lambda \| r\circ K\|}\|\partial_{\theta}(W^{\leq}+H)\|.
\end{align*}

By boundedness of $W^{\leq}$, $H$, $\overline{r\circ K}$, and their derivatives, we have
\[
\left|\frac{\partial}{\partial \theta}\widetilde{H}_i(\theta,s)\right|\leq C|s|^{N}.
\]
Above $C$ depends on $B$, but it will not depend on the choice of $H\in D^>$.
Similarly,
\begin{align*}
\frac{\partial}{\partial s}\widetilde{H}_i(\theta,s)=\partial_{\theta}H_i(\theta-\omega\overline{r\circ K}((W^{\leq}+H)(\theta,s)),se^{-\lambda\overline{r\circ K}((W^{\leq}+H)(\theta,s))})\cdot\\
\cdot(-\omega)D(\overline{r\circ K})((W^{\leq}+H)(\theta,s))\partial_{s}(W^{\leq}+H)(\theta,s)\phantom{AAA}\\
+\partial_{s}H_i(\theta-\omega \overline{r\circ K}((W^{\leq}+H)(\theta,s)),se^{-\lambda \overline{r\circ K}((W^{\leq}+H)(\theta,s))})\cdot\\
\cdot\left(1+s(-\lambda)D(\overline{r\circ K})((W^{\leq}+H)(\theta,s))\partial_{s}(W^{\leq}+H)(\theta,s)\right)e^{-\lambda \overline{r\circ K}((W^{\leq}+H)(\theta,s))}.
\end{align*}
Then,
\begin{align*}
\left|\frac{\partial}{\partial s}\widetilde{H}_i(\theta,s)\right|\leq&B|s|^{N-1}e^{-\lambda (N-1) \|\overline{r\circ K}\|}\left(1+|s||\lambda|\|D(\overline{r\circ K})\|e^{-\lambda \| \overline{r\circ K}\|}\|\partial_{s}(W^{\leq}+H)\|\right)\\
&+B|s|^{N}e^{-\lambda N \|\overline{r\circ K}\|}|\omega| \|D(\overline{r\circ K})\|\|\partial_{s}(W^{\leq}+H)\|.
\end{align*}
Since we have $|s|\leq1$, regularity of $W^{\leq}$ and $H$ we have
\[
\left|\frac{\partial}{\partial s}\widetilde{H}_i(\theta,s)\right|\leq C|s|^{N-1}.
\]
The $C$ will not depend on the choice of $H$ as long as $\|H\|_{L^>,N}\leq B$. The proof of the claim is then finished by induction.

Now we observe that we can bound the integrand in the operator $\Gamma^>$. 

Claim: There exists constant $C$, such that $\|Y(H,\theta, s, \varepsilon)\phi(s)\|_{L^>,N}\leq C$ when $\|H\|_{L^>,N}\leq B$.\\
Note that by definition of the cut-off function $\phi$, it suffices to consider $s\in[-1,1]$.
\begin{equation*}
Y^>(H,\theta, s,\varepsilon)=\overline{Y}((W^{\leq}+H)(\theta,s),\widetilde{(W^{\leq}+H)}(\theta,s),\varepsilon)-\sum^{N-1}_{i=0}\overline{Y}^i(\theta)s^i,
\end{equation*}
where 
\[
\overline{Y}^i(\theta)=\frac{1}{i!}\frac{\partial^i}{\partial s^i}(\overline{Y}((W^{\leq}+H)(\theta,s),\widetilde{(W^{\leq}+H)}(\theta,s),\varepsilon))|_{s=0}.
\]
One can add and subtract terms in above expression,
\begin{equation}\label{Y>est}
\begin{split}
Y^>(H,\theta, s,\varepsilon)=&\overline{Y}((W^{\leq}+H)(\theta,s),\widetilde{(W^{\leq}+H)}(\theta,s),\varepsilon)\\
&-\overline{Y}(W^{\leq}(\theta,s),\widetilde{W^{\leq}}(\theta,s,H),\varepsilon)\\
&+\overline{Y}(W^{\leq}(\theta,s),\widetilde{W^{\leq}}(\theta,s,H),\varepsilon)\\
&-\overline{Y}(W^{\leq}(\theta,s),W^{\leq}(\theta-\omega \overline{r\circ K}(W^{\leq}(\theta,s)),se^{-\lambda\overline{r\circ K}(W^{\leq}(\theta,s))}),\varepsilon)\\
&+\overline{Y}(W^{\leq}(\theta,s),W^{\leq}(\theta-\omega \overline{r\circ K}(W^{\leq}(\theta,s)),se^{-\lambda\overline{r\circ K}(W^{\leq}(\theta,s))}),\varepsilon)\\
&-\sum^{N-1}_{i=0}\overline{Y}^i(\theta)s^i,
\end{split}
\end{equation}

where we used the notation 
\[
\widetilde{W^{\leq}}(\theta,s;H)=W^{\leq}(\theta-\omega \overline{r\circ K}((W^{\leq}+H)(\theta,s)),se^{-\lambda\overline{r\circ K}((W^{\leq}+H)(\theta,s))}).
\]
We group the first two lines, the two lines in the middle, and the last two lines in \eqref{Y>est}, and denote them as $\ell_1$, $\ell_2$, and $\ell_3$, respectively. Then for $\ell_1$:

\begin{align*}
\ell_1=\int^1_0D_1\overline{Y}((1-t)W^{\leq}(\theta,s)+t(W^{\leq}+H)(\theta,s),\widetilde{(W^{\leq}+H)}(\theta,s),\varepsilon)H(\theta,s)dt\\
+\int^1_0D_2\overline{Y}(W^{\leq}(\theta,s),(1-t)\widetilde{W^{\leq}}(\theta,s;H)+t\widetilde{(W^{\leq}+H)}(\theta,s),\varepsilon)\widetilde{H}(\theta,s)dt
\end{align*}
By the regularity of $Y$, and $W^{\leq}$ and $\|H\|_{L^>,N}\leq B$, using that $\widetilde{H}$ satisfy \eqref{tildeh}, we know that $\|\ell_1\phi(s)\|_{L^>,N}\leq C$.

Similarly $\ell_2$ is
\begin{align*}
\int^1_0&D_2\overline{Y}(W^{\leq}(\theta,s),W^{\leq}(\theta-\omega \overline{r\circ K}((W^{\leq}+tH)(\theta,s)),se^{-\lambda\overline{r\circ K}((W^{\leq}+tH)(\theta,s))}),\varepsilon)\cdot\\
&[\partial_{\theta}W^{\leq}(\cdot)(-\omega)D(\overline{r\circ K})(\cdot)+\partial_{s}W^{\leq}(\cdot)se^{-\lambda\overline{r\circ K}(\cdot)}D(\overline{r\circ K})(\cdot)(-\lambda)]H(\theta,s)dt,
\end{align*}
Similar to $\ell_1$ case, we have that  $\|\ell_2\phi(s)\|_{L^>,N}\leq C$.

For the third line, notice that $\sum^{N-1}_{i=0}\overline{Y}^i(\theta)s^i$ is the Taylor expansion at $s=0$ for 
\begin{equation}\label{ylow}
\overline{Y}(W^{\leq}(\theta,s),W^{\leq}(\theta-\omega \overline{r\circ K}(W^{\leq}(\theta,s)),se^{-\lambda\overline{r\circ K}(W^{\leq}(\theta,s))}),\varepsilon),
\end{equation}

According to Taylor's Formula with remainder, see \cite{RaXue}, we just need to show that for $m\leq N$
\[
\frac{\partial^{N-m}}{\partial s^{N-m}}\frac{\partial^l}{\partial \theta^l}\frac{\partial^m}{\partial s^m}\eqref{ylow},
\]
and for $m>N$,
\[
\frac{\partial^m}{\partial s^m}\frac{\partial^l}{\partial \theta^l}(\ell_3),
\]
are bounded for all $\theta$, $|s|\leq 1$, and $l+m\leq L^>$. This is true if we assume that the lower order term has more regularity, more precisely, $L-1\geq L^>+N$. We will take $L^>=L-1-N$ to optimize regularity. Therefore, we have $\|\ell_3\phi(s)\|_{L^>,N}\leq C$, and the claim is proved. 

Hence, according to \eqref{op>}, if $m\leq N$, for small $\varepsilon$, we have that
\begin{equation}
|\partial^l_{\theta}\partial^m_s\Gamma^>_i(H)(\theta,s)|\leq \varepsilon\left|\int^{\infty}_0e^{-\lambda_0 t}C|s|^{N-m}e^{\lambda (N-m)t}e^{\lambda mt}dt\right|\leq B|s|^{N-m},
\end{equation}
if $m>N$, for small $\varepsilon$, we have that
\begin{equation}
|\partial^l_{\theta}\partial^m_s\Gamma^>_i(H)(\theta,s)|\leq \varepsilon\left|\int^{\infty}_0e^{-\lambda_0 t}Ce^{\lambda mt}dt\right|\leq B,
\end{equation}

Therefore, for small $\ep$, $\|\Gamma^>_i(H)\|_{L^>, N}\leq B$ when $\|H\|_{L^>, N}\leq B$.
\end{proof}
\begin{lemma}\label{lem:contraction>}
If $\varepsilon$ small enough, we have $\Gamma^>$ is a contraction in $\|\cdot\|_{0,N}$. 
\end{lemma}
\begin{proof}
Recall that $\|H\|_{0,N}=\sup_{(\theta,s)\in\T\times\R}|H(\theta,s)||s|^{-N}$. We consider
\begin{multline}
\Gamma^>(H)(\theta,s)-\Gamma^>(H')(\theta,s)\\
\qquad= -\varepsilon\int^{\infty}_0 \begin{pmatrix}
1&0\\
0&e^{-\lambda_0t}
\end{pmatrix}\left(Y^>(H,c(t), \varepsilon)-Y^>(H',c(t), \varepsilon)\right)\phi(se^{\lambda t})dt
\end{multline}
Given the low order terms, denote $W=W^{\leq}+H$ and $W'= W^{\leq}+H'$, we have
\begin{multline}
Y^>(H,c(t), \varepsilon)-Y^>(H',c(t), \varepsilon)\\
=\overline{Y}(W(c(t)),\widetilde{W}(c(t)),\varepsilon)-\overline{Y}(W'(c(t)),\widetilde{W'}(c(t)),\varepsilon).
\end{multline}

Note that for all $\theta$, $s$,
\begin{equation}
|W(\theta,s)-W'(\theta,s)|=|H(\theta,s)-H'(\theta,s)|\leq \|H-H'\|_{L^>,N}|s|^{N}.
\end{equation}

Then for $\widetilde{W}(\theta,s)-\widetilde{W'}(\theta,s)$, by adding and subtracting terms, we have for all $\theta$, $s$,
\begin{align*}
|\widetilde{W}(\theta,s)-\widetilde{W'}(\theta,s)|
=&\bigg|W(\theta-\omega\overline{r\circ K}(W(\theta,s)), se^{-\lambda \overline {r\circ K}(W(\theta,s))})\\
&\phantom{A}-W'(\theta-\omega\overline{r\circ K}(W'(\theta,s)), se^{-\lambda \overline {r\circ K}(W'(\theta,s))})\bigg|\\
\leq&\bigg|W(\theta-\omega\overline{r\circ K}(W(\theta,s)), se^{-\lambda \overline {r\circ K}(W(\theta,s))})\\
&\phantom{AA}-W'(\theta-\omega\overline{r\circ K}(W(\theta,s)), se^{-\lambda \overline {r\circ K}(W(\theta,s))})\bigg|\\
&+\bigg|W'(\theta-\omega\overline{r\circ K}(W(\theta,s)), se^{-\lambda \overline {r\circ K}(W(\theta,s))})\\
&\phantom{AA}-W'(\theta-\omega\overline{r\circ K}(W'(\theta,s)), se^{-\lambda \overline {r\circ K}(W(\theta,s))})\bigg|\\
&+ \bigg|W'(\theta-\omega\overline{r\circ K}(W'(\theta,s)), se^{-\lambda \overline {r\circ K}(W(\theta,s))})\\
&\phantom{AA}-W'(\theta-\omega\overline{r\circ K}(W'(\theta,s)), se^{-\lambda \overline {r\circ K}(W'(\theta,s))})\bigg|\\
\leq& M_1 \|H-H'\|_{0,N}|s|^N,
\end{align*}

where
\[
M_1=e^{-\lambda N\|\overline{r\circ K}\|}+(\|DW^{\leq}\|+B)\|D(\overline {r\circ K})\|(|\omega|+|\lambda||s|e^{-\lambda\|\overline {r\circ K}\|}).
\]
Then,
\[
|\Gamma^>(H)(\theta,s)-\Gamma^>(H')(\theta,s)|\leq \varepsilon\|H-H'\|_{0,N}|s|^N
\int^{\infty}_0 e^{(\lambda N-\lambda_0)t}M\phi(se^{\lambda t})dt,
\]
where 
\[
M=\|D_1\overline{Y}\|+\|D_2\overline{Y}\|M_1.
\]
Now, notice that by definition of $D^1$, we have that $\lambda\in [\frac{4\lambda_0}{3},\frac{2\lambda_0}{3}]$, then $\lambda N-\lambda_0<0$ if $N\geq2 $. Under this assumption, we have for all $\theta$, $s$,
\begin{equation*}
|\Gamma^>(H)(\theta,s)-\Gamma^>(H')(\theta,s)|\leq-\frac{\varepsilon M}{\lambda N-\lambda_0}\|H-H'\|_{0,N}|s|^N.
\end{equation*}
If $\varepsilon$ is small enough, we have for all $\theta$, $s$,
\[
|\Gamma^>(H)(\theta,s)-\Gamma^>(H')(\theta,s)|\leq\mu\|H-H'\|_{0,N}|s|^N.
\]
Hence for small enough $\ep$,
\[
\|\Gamma^>(H)-\Gamma^>(H')\|_{0,N}\leq\mu\|H-H'\|_{0,N},
\]
$\Gamma^>$ is a contraction. Note that smallness condition for $\ep$ depends on $N$, $B^j$, $j=0, \dotsc, N-1$, $B$, $\omega$, $\lambda$, $\overline{Y}$, and $\overline{r\circ K}$.
\end{proof}

Now for any initial guess $W^{<,0}$, the sequence $(\Gamma^>)^n(W^{>,0})$ in the function space $D^>$, will converge pointwise to a function $W^>$, which is a fixed point of $\Gamma^>$. By Lemma \ref{lem:Lan}, we know that $W^>$ is $(L^>-1)$ times differentiable, with $(L^>-1)$-th derivative Lipschitz. 

It remains to do the error analysis in this case. Notice that

\begin{equation*}
E^>(\theta,s)=(\omega\partial_{\theta}+s\lambda\partial_s)W^{>,0}(\theta,s)-\begin{pmatrix}
0\\\lambda_0W^{>,0}_2(\theta, s)
\end{pmatrix}-\varepsilon Y^>(W^{>,0},\theta, s, \varepsilon)\phi(s),
\end{equation*}
along the characteristics, we have
\begin{align*}
E^>(c(t))=(\omega&\partial_{\theta}+se^{\lambda t}\lambda\partial_s)W^{>,0}(c(t))-\begin{pmatrix}
0\\\lambda_0W^{>,0}_2(c(t))
\end{pmatrix}\\&-\varepsilon Y^>(W^{>,0},c(t), \varepsilon)\phi(se^{\lambda t}).
\end{align*}

Hence,
\begin{equation*}
\Gamma^>(W^{>,0})(\theta,s)-W^{>,0}(\theta,s)=\int^{\infty}_0\begin{pmatrix}
1&0\\0&e^{-\lambda_0 t}
\end{pmatrix} E^>(c(t))dt.
\end{equation*}

Based on proof of Lemma \ref{lem:prob>}, we know that $\|E^>\|_{0,N}$ is bounded, therefore, for the maximum norm,
\begin{equation*}
\|\Gamma^>(W^{>,0})-W^{>,0}\|\leq \frac{1}{\lambda_0-\lambda N}\|E^>\|_{0,N}|s|^N,
\end{equation*}
and then
\begin{equation}\label{err>}
\|W^{>}-W^{>,0}\|\leq \frac{1}{1-\mu}\|\Gamma^>(W^{>,0})-W^{>,0}\|\leq\frac{1}{(1-\mu)(\lambda_0-\lambda N)}\|E^>\|_{0,N}|s|^N.
\end{equation}

If we take account of error estimations in \eqref{err0},\eqref{err1}, \eqref{errj}, and \eqref{err>}, we see that $l=0$ case of \eqref{aprw} is proved. Inequalities in \eqref{aprw} for $l\neq0$ is obtained using interpolation inequalities.

\subsection{Proof of Theorem~\ref{thm:smooth} and Theorem ~\ref{thm:smoothh}} 

The proof of  Theorem~\ref{thm:smooth}  and Theorem ~\ref{thm:smoothh} are obtained 
by just considering the functions $W^j_\eta$ as functions of 
two variables $\tilde W^j(\eta, \theta)$.  We 
can straightforwardly lift the operators $\Gamma^0$, $\Gamma^1$, and $\Gamma^j$ defined in 
\eqref{op0}, \eqref{op1}, and \eqref{opj}  to operators acting on 
functions of two variables.  We denote these operators
acting on two-variable functions by $\tilde \Gamma^0$, $\tilde \Gamma^1$, and $\tilde \Gamma^j$, respectively. At the same time, we lift the operator $\Gamma^>$ to an operator acting on functions of three variables, denoted as $\tilde \Gamma^>$.

To prove Theorem \ref{thm:smooth}, given a function 
 $\tilde W^0(\eta, \theta)$ of the variables $\eta, \theta$, we treat 
$\eta$ as a parameter and take into account that now, $Y$ and $r$ depend
also on $\eta$, in a smooth way. 

We use the same strategy as in the proof of 
Theorem~\ref{thm:zero}.  We first show the 
propagated bounds, similar to Lemma \ref{lem:prob0}, and then, show 
that the operator is a contraction under
a distance given by the  $C^0$ norm of the two-variable functions and the 
distance on the $\omega$, similar to Lemma \ref{lem:contraction0}. The distance here is quite analogue to the distance defined in
\eqref{dist}. Then, 
the desired result, Theorem~\ref{thm:smooth} 
 follows by an application of 
Lemma~\ref{lem:Lan}. 

The key to the propagated bounds is to show that 
if $\| \tilde W \|_{L+\text{Lip}} \le \tilde B^0$, for $\eps < \eps_0$, 
we have that the $C^{L+Lip}$ norm of the function components of 
$\tilde \Gamma^0( \tilde W)$ is also smaller or equal 
than $\tilde B^0$. 
This proof is rather straightforward and identical to 
the proof as before, because if we apply 
Fa\'a di Bruno formula, we obtain that the derivatives
of order up to $L$ of $\tilde \Gamma( \tilde W^0)$, 
are polynomials in the derivatives of $\tilde W^0$ of
order up to $L$ and the coefficients are just derivatives of 
$Y$, $r$ and combinatorial coefficients.  Similarly, we can 
estimate the Lipschitz constants because upper bounds for the
Lipschitz constants satisfy an  analogue of Fa\'a di Bruno formula.

To obtain the proof of the contraction in $C^0$, we just need to observe
that the proof of the contraction in Theorem~\ref{thm:zero}  only uses 
very few properties of $Y, r$. The properties hold uniformly for all 
$\eta$. One can obtain the contraction in the uniform 
norm on both variables. 

Analogous arguments as above for the operators $\tilde \Gamma^j$ and $\tilde \Gamma^>$, using similar methods in Sections \ref{ssc:pw1}, \ref{ssc:pwj}, \ref{ssc:phot}, complete the proof for Theorem \ref{thm:smoothh}. 

\section*{Acknowledgements}
The authors would like to thank A. Humphries and R. Calleja for several suggestions. The authors also thank A. Haro, T. M-Seara, and C. Zeng for discussions.
\bibliographystyle{alpha}
\bibliography{ref}

\newcommand{\etalchar}[1]{$^{#1}$}
\begin{thebibliography}{DvGVLW95}

\bibitem[AR67]{AbrahamR67}
Ralph Abraham and Joel Robbin.
\newblock {\em Transversal mappings and flows}.
\newblock An appendix by Al Kelley. W. A. Benjamin, Inc., New York-Amsterdam,
  1967.

\bibitem[Car81]{Carr}
Jack Carr.
\newblock {\em Applications of centre manifold theory}, volume~35 of {\em
  Applied Mathematical Sciences}.
\newblock Springer-Verlag, New York-Berlin, 1981.

\bibitem[CCdlL19]{Livia19}
Alfonso Casal, Livia Corsi, and Rafael de~la Llave.
\newblock Lindstedt series for sdde.
\newblock {\em --}, --(-):--, 2019.
\newblock https://arxiv.org/abs/1910.04808.

\bibitem[CF80]{CasalFreed}
A.~Casal and M.~Freedman.
\newblock A {P}oincar\'{e}-{L}indstedt approach to bifurcation problems for
  differential-delay equations.
\newblock {\em IEEE Trans. Automat. Control}, 25(5):967--973, 1980.

\bibitem[Chi03]{Chiconedelay}
Carmen Chicone.
\newblock Inertial and slow manifolds for delay equations with small delays.
\newblock {\em J. Differential Equations}, 190(2):364--406, 2003.

\bibitem[CHK17]{Hum}
R.~C. Calleja, A.~R. Humphries, and B.~Krauskopf.
\newblock Resonance phenomena in a scalar delay differential equation with two
  state-dependent delays.
\newblock {\em SIAM J. Appl. Dyn. Syst.}, 16(3):1474--1513, 2017.

\bibitem[CJS63]{CurrieS}
D.~G. Currie, T.~F. Jordan, and E.~C.~G. Sudarshan.
\newblock Relativistic invariance and {H}amiltonian theories of interacting
  particles.
\newblock {\em Rev. Modern Phys.}, 35:350--375, 1963.

\bibitem[dlLO99]{LO99}
R.~de~la Llave and R.~Obaya.
\newblock Regularity of the composition operator in spaces of {H}\"{o}lder
  functions.
\newblock {\em Discrete Contin. Dynam. Systems}, 5(1):157--184, 1999.

\bibitem[Dri84]{Driver}
R.~D. Driver.
\newblock A neutral system with state-dependent delay.
\newblock {\em J. Differential Equations}, 54(1):73--86, 1984.

\bibitem[DvGVLW95]{Diek}
Odo Diekmann, Stephan~A. van Gils, Sjoerd~M. Verduyn~Lunel, and Hans-Otto
  Walther.
\newblock {\em Delay equations}, volume 110 of {\em Applied Mathematical
  Sciences}.
\newblock Springer-Verlag, New York, 1995.
\newblock Functional, complex, and nonlinear analysis.

\bibitem[GG73]{GolubitskyG73}
M.~Golubitsky and V.~Guillemin.
\newblock {\em Stable mappings and their singularities}.
\newblock Springer-Verlag, New York-Heidelberg, 1973.
\newblock Graduate Texts in Mathematics, Vol. 14.

\bibitem[Gim19]{JoanTh}
Joan Gimeno.
\newblock {\em Effective methods for recurrence solutions in delay differential
  equations}.
\newblock PhD thesis, Universitat de Barcelona, 2019.

\bibitem[GMJ17]{GroothedeMJ17}
C.~M. Groothedde and J.~D. Mireles~James.
\newblock Parameterization method for unstable manifolds of delay differential
  equations.
\newblock {\em J. Comput. Dyn.}, 4(1-2):21--70, 2017.

\bibitem[Guc75]{Guc}
J.~Guckenheimer.
\newblock Isochrons and phaseless sets.
\newblock {\em J. Math. Biol.}, 1(3):259--273, 1974/75.

\bibitem[GYdlL19]{Joan19}
Joan Gimeno, Jiaqi Yang, and Rafael de~la Llave.
\newblock Numerical computation of periodic orbits and isochrons for
  state-dependent delay perturbation of an ode in the plane.
\newblock {\em --}, --(-):--, 2019.

\bibitem[Had98]{Had98}
J.~Hadamard.
\newblock Sur le module maximum d'une fonction et de ses derives.
\newblock {\em Bull. Soc. Math. France}, 42:68--72, 1898.

\bibitem[Hal77]{Hale}
Jack Hale.
\newblock {\em Theory of functional differential equations}.
\newblock Springer-Verlag, New York-Heidelberg, second edition, 1977.
\newblock Applied Mathematical Sciences, Vol. 3.

\bibitem[HBC{\etalchar{+}}16]{HumphriesBCHS16}
A.~R. Humphries, D.~A. Bernucci, R.~C. Calleja, N.~Homayounfar, and M.~Snarski.
\newblock Periodic solutions of a singularly perturbed delay differential
  equation with two state-dependent delays.
\newblock {\em J. Dynam. Differential Equations}, 28(3-4):1215--1263, 2016.

\bibitem[HdlL13]{HL13}
Gemma Huguet and Rafael de~la Llave.
\newblock Computation of limit cycles and their isochrons: fast algorithms and
  their convergence.
\newblock {\em SIAM J. Appl. Dyn. Syst.}, 12(4):1763--1802, 2013.

\bibitem[HDlL16]{HR2}
Xiaolong He and Rafael De~la Llave.
\newblock Construction of quasi-periodic solutions of state-dependent delay
  differential equations by the parameterization method i: Finitely
  differentiable, hyperbolic case.
\newblock {\em Journal of Dynamics and Differential Equations}, 02 2016.

\bibitem[HdlL17]{HR}
Xiaolong He and Rafael de~la Llave.
\newblock Construction of quasi-periodic solutions of state-dependent delay
  differential equations by the parameterization method {I}: {F}initely
  differentiable, hyperbolic case.
\newblock {\em J. Dynam. Differential Equations}, 29(4):1503--1517, 2017.

\bibitem[HKWW06]{hart}
Ferenc Hartung, Tibor Krisztin, Hans-Otto Walther, and Jianhong Wu.
\newblock Functional differential equations with state-dependent delays: theory
  and applications.
\newblock In {\em Handbook of differential equations: ordinary differential
  equations. {V}ol. {III}}, Handb. Differ. Equ., pages 435--545.
  Elsevier/North-Holland, Amsterdam, 2006.

\bibitem[HT97]{HartungTuri}
Ferenc Hartung and Janos Turi.
\newblock On differentiability of solutions with respect to parameters in
  state-dependent delay equations.
\newblock {\em J. Differential Equations}, 135(2):192--237, 1997.

\bibitem[HVL93]{HaleLunel}
Jack~K. Hale and Sjoerd~M. Verduyn~Lunel.
\newblock {\em Introduction to functional-differential equations}, volume~99 of
  {\em Applied Mathematical Sciences}.
\newblock Springer-Verlag, New York, 1993.

\bibitem[KL12]{KissL12}
G\'{a}bor Kiss and Jean-Philippe Lessard.
\newblock Computational fixed-point theory for differential delay equations
  with multiple time lags.
\newblock {\em J. Differential Equations}, 252(4):3093--3115, 2012.

\bibitem[KL17]{KissL17a}
Gabor Kiss and Jean-Philippe Lessard.
\newblock Rapidly and slowly oscillating periodic solutions of a delayed van
  der {P}ol oscillator.
\newblock {\em J. Dynam. Differential Equations}, 29(4):1233--1257, 2017.

\bibitem[KM97]{KrieglM97}
Andreas Kriegl and Peter~W. Michor.
\newblock {\em The convenient setting of global analysis}, volume~53 of {\em
  Mathematical Surveys and Monographs}.
\newblock American Mathematical Society, Providence, RI, 1997.

\bibitem[Kol49]{Kol}
A.~Kolmogoroff.
\newblock On inequalities between the upper bounds of the successive
  derivatives of an arbitrary function on an infinite interval.
\newblock {\em Amer. Math. Soc. Translation}, 1949(4):19, 1949.

\bibitem[Lan73]{Lan}
Oscar~E Lanford.
\newblock Bifurcation of periodic solutions into invariant tori: the work of
  ruelle and takens.
\newblock In {\em Nonlinear problems in the physical sciences and biology},
  pages 159--192. Springer, 1973.

\bibitem[LdlL09]{LiL}
Xuemei Li and Rafael de~la Llave.
\newblock Construction of quasi-periodic solutions of delay differential
  equations via {KAM} techniques.
\newblock {\em J. Differential Equations}, 247(3):822--865, 2009.

\bibitem[LdlL10]{RaXue}
Xuemei Li and Rafael de~la Llave.
\newblock Convergence of differentiable functions on closed sets and remarks on
  the proofs of the ``converse approximation lemmas''.
\newblock {\em Discrete Contin. Dyn. Syst. Ser. S}, 3(4):623--641, 2010.

\bibitem[Les10]{Lessard10}
Jean-Philippe Lessard.
\newblock Recent advances about the uniqueness of the slowly oscillating
  periodic solutions of {W}right's equation.
\newblock {\em J. Differential Equations}, 248(5):992--1016, 2010.

\bibitem[MKW14]{MagpantayKW14}
F.~M.~G. Magpantay, N.~Kosovali\'{c}, and J.~Wu.
\newblock An age-structured population model with state-dependent delay:
  derivation and numerical integration.
\newblock {\em SIAM J. Numer. Anal.}, 52(2):735--756, 2014.

\bibitem[MNnO17]{Obaya}
Ismael Maroto, Carmen N\'{u}\~{n}ez, and Rafael Obaya.
\newblock Exponential stability for nonautonomous functional differential
  equations with state-dependent delay.
\newblock {\em Discrete Contin. Dyn. Syst. Ser. B}, 22(8):3167--3197, 2017.

\bibitem[MPN11]{MPN}
John Mallet-Paret and Roger~D. Nussbaum.
\newblock Stability of periodic solutions of state-dependent delay-differential
  equations.
\newblock {\em J. Differential Equations}, 250(11):4085--4103, 2011.

\bibitem[MPNP94]{MPNP}
John Mallet-Paret, Roger~D. Nussbaum, and Panagiotis Paraskevopoulos.
\newblock Periodic solutions for functional-differential equations with
  multiple state-dependent time lags.
\newblock {\em Topol. Methods Nonlinear Anal.}, 3(1):101--162, 1994.

\bibitem[Sie17]{Sieber}
Jan Sieber.
\newblock Local bifurcations in differential equations with state-dependent
  delay.
\newblock {\em Chaos}, 27(11):114326, 12, 2017.

\bibitem[Sij85]{Sijbrand}
Jan Sijbrand.
\newblock Properties of center manifolds.
\newblock {\em Trans. Amer. Math. Soc.}, 289(2):431--469, 1985.

\bibitem[Wal03]{Walt}
Hans-Otto Walther.
\newblock The solution manifold and $c^1$-smoothness for differential equations
  with state-dependent delay.
\newblock {\em Journal of Differential Equations}, 195(1):46 -- 65, 2003.

\bibitem[Win75]{Win}
A.~T. Winfree.
\newblock Patterns of phase compromise in biological cycles.
\newblock {\em J. Math. Biol.}, 1(1):73--95, 1974/75.

\end{thebibliography}

\end{document}